\documentclass[reqno,12pt]{amsart}
\usepackage[all,line,poly,rotate,ps,dvips]{xy}
\SelectTips{cm}{}

\usepackage{amssymb,eucal,graphicx}
\usepackage{enumerate}

\numberwithin{equation}{section}

\newtheorem{theorem}[equation]{Theorem}

\newtheorem{corollary}[equation]{Corollary}
\newtheorem{claim}[equation]{Claim}
\newtheorem{lemma}[equation]{Lemma}
\newtheorem{proposition}[equation]{Proposition}

\newtheorem{ass}[equation]{Assumptions}

\theoremstyle{definition}

\newtheorem{definition}[equation]{Definition}
\newtheorem{remark}[equation]{Remark}

\newtheorem{example}[equation]{Example}

\newtheorem{constr}[equation]{Construction}

\theoremstyle{remark}

\newcommand{\C}{\mathbb{C}}
\newcommand{\ZZ}{\mathbb{Z}}

\newcommand{\Pp}{\mathbb{P}}

\newcommand{\Ff}{\mathcal{F}}

\newcommand{\Ef}{\mathcal{E}}

\newcommand{\HH}{\mathcal{H}}

\newcommand{\Ii}{\mathcal{I}}

\newcommand{\N}{\mathcal{N}}
\newcommand{\M}{\mathcal{M}}

\newcommand{\Oc}{\mathcal{O}}

\begin{document}
\title[Non-special scrolls with general moduli]%
{Non-special scrolls with general moduli}

\author{Alberto Calabri, Ciro Ciliberto, Flaminio Flamini, Rick Miranda}

\email{calabri@dmsa.unipd.it}
\curraddr{Dipartimento di Metodi e Modelli Matematici per le
Scienze Applicate, Universit\`a degli Studi di Padova\\
Via Trieste, 63 - 35121 Padova \\Italy}

\email{cilibert@mat.uniroma2.it} \curraddr{Dipartimento di
Matematica, Universit\`a degli Studi di Roma Tor Vergata\\ Via
della Ricerca Scientifica - 00133 Roma \\Italy}

\email{flamini@mat.uniroma2.it} \curraddr{Dipartimento di
Matematica, Universit\`a degli Studi di Roma Tor Vergata\\ Via
della Ricerca Scientifica - 00133 Roma \\Italy}

\email{Rick.Miranda@ColoState.Edu} \curraddr{Department of
Mathematics, 101 Weber Building, Colorado State University
\\ Fort Collins, CO 80523-1874 \\USA}

\thanks{
\\
The first three authors are members of G.N.S.A.G.A.\ at
I.N.d.A.M.\ ``Francesco Severi''.}

\begin{abstract} In this paper we study smooth, non-special scrolls $S$ of degree $d$,
genus $g \geq 0$, with general moduli. In particular, we study the scheme
of unisecant curves of a given degree on $S$.
Our approach is mostly based on degeneration techniques.
\end{abstract}

\maketitle


\section{Introduction}\label{S:Intro}

It is well-known that the study of vector bundles on curves is equivalent
to the one of scrolls in projective space. In the present article, we will
mostly take the projective point of view, together with degeneration techniques,
in order to study smooth, non-special scroll surfaces of degree $d$, sectional
genus $g \geq 0$, with general moduli,
which are linearly normal in $\Pp^R$, $R = d- 2g +1$. However we will bridge
this approach with the vector-bundle one showing how
projective geometry and degenerations can be used in order to improve
some known results about rank-two vector bundles and to obtain some new ones (cf.\ also \cite{CCFMBN}).

The first three sections of the paper basically contain some folklore, which we think
will be useful for a possible reader. In
\S\ \ref{S:NotPre} and \ref{S:PreRes} we fix notation and terminology and
recall preliminary results on scrolls. In \S\ \ref{S:VectBun} we introduce
the vector bundle setting.

If $d\geq 2g+3+\min\{1,g-1\}$,
such scrolls fill up a unique component $\HH_{d,g}$ of the Hilbert scheme of surfaces
in $\Pp^R$ which dominates ${\mathcal M}_g$ (this result is essentially contained in 
\cite{APS}; see \cite{CCFMLincei} for an explicit statement and different proofs).
Let $[S] \in \HH_{d,g}$ be a general point, such that $S \cong \Pp(\Ff)$, where
$\Ff$ is a very ample rank-two vector bundle of degree $d$ on $C$, a curve of genus $g$ with
general moduli, and $S$ is embedded in $\Pp^R$ via the global sections of $\Oc_{\Pp(\Ff)} (1)$.
In \S\ \ref{S:HNSS}, we first recall that if $g \geq 2$ and $S$ is general, then
$\Ff$  is stable (in case $g=1$ there is
a slightly different result; cf.\ Theorem \ref{thm:scrollfib1} and Remark \ref{rem:deven}. This 
result is contained in \cite{APS}. We give here a short, independent proof).
This suggests that $\HH_{d,g}$ plays, in the projective geometry
setting, a role analogous to the one of
the moduli stack of semistable rank-two vector bundles of degree $d$ over ${\mathcal M}_g$. We
discuss in Remarks \ref{rem:hilbparam2} and \ref{rem:CxP1} a few examples showing that $\HH_{d,g}$ contains
points corresponding to unstable bundles as well as to strictly semistable
products of the type $C \times \Pp^1$.  We finish \S\ \ref{S:HNSS} by describing two constructions closely
related to the ones in \cite{CCFMLincei} (cf.\ Constructions \ref{const:T} and \ref{const:Y})
which prove that $\HH_{d,g}$ contains smooth points corresponding
to some reducible scrolls. The results in 
\cite{APS} also imply that  $\HH_{d,g}$ contains  points corresponding
to reducible scrolls. However, the ones that we need to consider in this paper 
are different from those in \cite{APS} and therefore we have to introduce them here. Note that 
in \cite{CCFMLincei} one proves that $\HH_{d,g}$ contains points corresponding to surfaces 
which are reducible in suitable  unions of planes, thus solving an old problem posed by G. Zappa (cf. \cite{Zapp}).  
These planar degenerations however are not used here. 

In \S\ \ref{S:unisec} we consider the scheme ${\rm Div}_S^{1,m}$ parametrizing
unisecant curves of given degree $m$ on $S$ (cf.\ Definition \ref{def:ghio0}).
By using degeneration arguments involving the above constructions,
we prove Theorem \ref{thm:ganzo2}, which says that $S$ is a {\em general ruled surface} in the sense of Ghione
in \cite[Definition 6.1]{Ghio} (cf.\ Definition \ref{def:ghio} below). Namely
\begin{enumerate}[(i)]

\item $ {\rm dim}({\rm Div}_S^{1,m}) = d_m := {\rm max} \{ -1, \; 2 m - d - g + 1\}$;

\item ${\rm Div}_S^{1,m}$ is smooth, for any $m$ such that $d_m \geq 0$;

\item ${\rm Div}_S^{1,m}$ is irreducible, for any $m$ such that $d_m > 0$.
\end{enumerate}This, in particular, gives effective existence results for general ruled surfaces (whereas
\cite[Th\'eor\`eme 7.1]{Ghio} is only asymptotic, cf.\ Theorem \ref{thm:Ghio} below). The idea 
of using degenerations to study unisecants is already present in \cite{APS}, where however it is 
used only to prove the stability of the rank-two vector bundle determining the general point 
of $\HH_{d,g}$.

In \S\ \ref{S:enumerative} we make some applications proving a few enumerative properties
of ${\rm Div}_S^{1,m}$. In Theorem \ref{thm:ganzo} we give a new and easy proof
of a result of Ghione (cf.\ \cite[Th\'eor\`eme 6.4 and 6.5]{Ghio}), which provides also an effective
version of it. Then, we apply Theorem \ref{thm:ganzo} to compute the so called
{\em index} of ${\rm Div}_S^{1,m}$, thus giving a new proof of a result of Corrado Segre
(cf.\ Proposition \ref{prop:index}). In \S\ \ref{SS:3} we study the monodromy action on
${\rm Div}_S^{1,m}$ in case this is finite, proving that the monodromy acts as the full symmetric
group. Finally, in \S\ \ref{SS:5} we extend the result in  \cite{Gro} and \cite[Example 3.2]{Oxb}, by
computing the genus of the curve parametrizing those divisors in
${\rm Div}_S^{1,m}$ passing through $d_m - 1$ general points of $S$.

This paper has to be regarded as the continuation of a project initiated with
\cite{CCFMLincei}.  In \cite{CCFMBN}
we  make further application of the ideas
contained in \cite{CCFMLincei} and in the present paper to the Brill-Noether theory
of sub-line bundles of rank-two vector bundles on curves. We will devote a forthcoming article to the case
of special scrolls.

\section{Notation and preliminaries}\label{S:NotPre}

In this section we will fix notation and general assumptions. As a general warning, if there is no
danger of confusion, we will identify line bundles and divisors.

Let $C$ be a smooth, projective curve
of genus $g \geq 0$ and let $\rho: F \to  C$ be a {\em
geometrically ruled surface} on $C$, i.e.\ $F = \Pp(\Ff)$, with $\Ff$ a
rank-two vector bundle on $C$.
For us, a rank $r$ vector bundle is a locally free sheaf of rank $r$.

In this paper, we shall make the following:

\begin{ass}\label{ass:1} With notation as above,
\begin{itemize}
\item[(1)] the rank-two vector bundle $\Ff$ is of {\em degree} $\deg(\Ff) := \deg(\det(\Ff)) = d$;
\item [(2)] $h^0(C, \Ff) = R+1$, for some $R \geq 3$;
\item[(3)] the complete linear system $|\Oc_F(1)|$ is base-point-free (therefore
the general element is a smooth, irreducible curve on $F$) and the morphism
$$\Phi: F \to \Pp^R$$induced by $|\Oc_F(1)|$ is birational to the image.
\end{itemize}
\end{ass}

\begin{definition}\label{def:scroll} The surface
$$\Phi(F) :=S \subset \Pp^R$$is said to be a {\em scroll of degree
$d$ and of (sectional) genus $g$}, and  $S$ is called the {\em
scroll determined by the pair} $(\Ff, C)$. Note that $S$ is smooth if and only if $\Ff$ is very ample.
If $S$ is not smooth, then $F$ is its minimal desingularization.

For any $x \in C$, let $f_x := \rho^{-1}(x) \cong \Pp^1$. The general fibre of $\rho$ will be denoted by $f$.
For any $x \in C$, the line $l_x := \Phi(f_x)$ is called a {\em ruling} of
$S$ (the general ruling of $S$ will be denoted by $l = \Phi(f)$). By abusing terminology, the
family $\{ l_x \}_{ x \in C}$ is also called the {\em ruling} of $S$.
\end{definition}

For further details on ruled surfaces, we refer the reader to
\cite{GH}, \cite[\S\,V]{Ha}, \cite{APS}, \cite{GP1}, \cite{GP2}, \cite{GP3}, \cite{Ghio},
\cite{GS}, \cite{GiSo}, \cite{Ma}, \cite{MN}, \cite{Na}, \cite{Seg}, \cite{Sev} and \cite{Sev1}.

Let $F \stackrel{\rho}{\to} C$ be as above. There is a section $i : C \hookrightarrow F$,
whose image we denote by $H$, such that $\Oc_F(H) = \Oc_F(1)$. Then

\begin{equation}\label{eq:picF}
{\rm Pic}(F) \cong \ZZ[\Oc_F(H)] \oplus \rho^*({\rm Pic}(C)).
\end{equation}Moreover,
\begin{equation}\label{eq:numF}
{\rm Num} (F) \cong \ZZ \oplus \ZZ,
\end{equation}generated by the classes of $H$ and $f$, satisfying $Hf=1$, $f^2 = 0$
(cf.\ \cite[\S\,5, Prop.\ 2.3]{Ha}). If
$\underline{d}\in {\rm Div}(C)$,  we denote by $\underline{d} f$ the divisor $\rho^*(\underline{d})$
on $F$. A similar notation will be used when $\underline{d}\in {\rm Pic}(C)$.

From \eqref{eq:picF} and \eqref{eq:numF}, any element of ${\rm Pic}(F)$ corresponds to a divisor on $F$
of the form$$nH + \underline{d} f, \; n \in \ZZ,  \; \underline{d}\in {\rm Pic}(C).$$As an element of ${\rm Num}(F)$
this is $$nH + d f, \; n \in \ZZ,  \; d= \deg(\underline{d}) \in \ZZ.$$

\begin{definition}\label{def:unisec} For any $n \geq 0$ and for any $\underline{d} \in {\rm Div}(C)$, the linear system
$|nH + \underline{d} f|$, if not empty, is said to be {\em n-secant} to the fibration
$F \stackrel{\rho}{\to} C$ since its general element meets $f$ at $n$ points.

For any $\underline{d} \in {\rm Div}(C)$ such that $|H + \underline{d} f| \neq \emptyset$,
any $B \in |H + \underline{d} f|$ is called a {\em unisecant curve} to the fibration $F \stackrel{\rho}{\to} C$
(or simply of $F$).

Any irreducible unisecant curve $B$ of $F$ is smooth and is called a
{\em section} of $F$.
\end{definition}

Recall that there is a one-to-one correspondence between sections $B$ of $F$ and surjections
$\Ff \twoheadrightarrow L$, with $L =L_B$ a line bundle on $C$ (cf.\ \cite[\S\;V, Prop.\ 2.6 and 2.9]{Ha}).
Then, one has an exact sequence
\begin{equation}\label{eq:Fund}
0 \to N \to \Ff \to L \to 0,
\end{equation}where $N$ is a line bundle on $C$.

Note that the surjection in \eqref{eq:Fund} induces the inclusion $B = \Pp(L) \subset \Pp(\Ff) = F$. If $L =\Oc_C( \underline{m}) $, with
$\underline{m} \in {\rm Div}^m(C)$, then $m = HB$ and $B \sim H + ( \underline{m} - \det(\Ff)) f $.

For example, if $B \in |H|$, the associated exact sequence is
\begin{equation}\label{eq:sfiga1}
0 \to \Oc_C \to \Ff \to \det(\Ff) \to 0,
\end{equation}where the map $\Oc_C \hookrightarrow \Ff$ gives a global section
of $\Ff$ which corresponds to the global section of $\Oc_F(1)$ vanishing on $B$.

With $B$ and $F$ as in Definition \ref{def:unisec}, from \eqref{eq:Fund}
\begin{equation}\label{eq:Ciro410b}
\Oc_B(B) \cong N^{\vee} \otimes L
\end{equation} (cf.\ \cite[\S\,5]{Ha}). In particular,
\begin{equation}\label{eq:Ciro410}
B^2 = \deg(L) - \deg(N) = d - 2 \, \deg(N) = 2m - d.
\end{equation}

Similar considerations hold if $B_1$ is a (reducible) unisecant curve of $F$ such that $H B_1 = m$.
Indeed, there exists a section $B \subset F$ and an effective divisor $\underline{a} \in {\rm Div}(C)$,
$ a:= \deg(\underline{a})$, such that$$B_1 = B + \underline{a} f,$$where $BH = m-a$. In particular there exists a
line bundle $L = L_B$ on $C$, with $\deg(L) = m-a$, fitting in a sequence like \eqref{eq:Fund}.
Consider the evaluation map
$$ev: \Ff \to \Oc_{\underline{a}}^{\oplus 2}.$$In a local trivialization around the points in $\underline{a}$
$\Ff$ splits as the sum of $N$ and $L$. Therefore, a local section $s$ of $\Ff$ around  $\underline{a}$,
can be considered as a pair $(s_1,s_2)$ where $s_1$ (respectively $s_2$) is a local section
of $N$ (respectively of $L$).  Then, $ev (s) = (ev(s_1), ev(s_2))$, where we denoted by
$ev$ also the obvious evaluation maps for $N$ and $L$. By projecting onto the diagonal of
$\Oc_{\underline{a}}^{\oplus 2}$ we have a surjection
$$\Ff \to \Oc_{\underline{a}}$$and therefore also a surjection of
$$\Ff \to L \oplus \Oc_{\underline{a}}$$and it is clear now that it fits into the exact sequence
\begin{equation}\label{eq:Fund2}
0 \to N \otimes \Oc_C( - \underline{a}) \to \Ff \to L \oplus \Oc_{\underline{a}} \to 0.
\end{equation}As above, the surjection in \eqref{eq:Fund2} induces the inclusion
$B_1 \subset F$.

\begin{definition}\label{def:direct} Let $S$ be a scroll of degree $d$ and
genus $g$ as in Definition \ref{def:scroll} corresponding to the  pair $(\Ff, C)$.
Let $B \subset F$ be a section and  $L $ as in \eqref{eq:Fund}.

Let $\Gamma := \Phi(B) \subset S$. If $\Phi|_B$ is birational to the image, then
$\Gamma$ is called  a {\em section} of the scroll $S$ (the classical terminology for $\Gamma$
was {\em directrix} of $S$, cf.\ e.g.\ \cite[Defn.\ 1.9]{GP2}) .

We will say that the pair
$(S,\Gamma)$ is {\em associated with} \eqref{eq:Fund} and that $\Gamma$
{\em corresponds to the line bundle} $L$ on $C$.

If $m = \deg(L)$, then $\Gamma$ is a {\em section of degree} $m$
of $S$; moreover, $$\Phi|_B : B \cong C \to \Gamma$$is determined by the linear series $\Lambda \subseteq |L|$, which
is the image of the map$$H^0(\Ff) \to H^0(L).$$

More generally, if $B_1 \subset F$ is a (reducible) unisecant curve and $\Phi|_{B_1}$ is birational to the image,
then we call $\Phi(B_1) = \Gamma_1$ a {\em unisecant curve of degree} $m$ of $S$. Note that such a curve
is the union of a section and of some rulings.

As above, the pair $(S, \Gamma_1)$ corresponds to a sequence of the type \eqref{eq:Fund2}.
\end{definition}

In general, the map $\Phi|_B$ may well be not an isomorphism,
not even birational to the image (cf.\ Example \ref{ex:contoflam}); indeed, $\Phi|_B$
can even contract $B$ to a point if
$L \cong \Oc_C$, in which case $S$ is a cone (cf.\ Lemma \ref{lem:segre}).

When $g= 0$ we have rational scrolls and these are well-known (see
e.g.\ \cite{GH}). Thus, from now on, we shall  focus on the case $g
\geq 1$.

\section{Preliminary results on scrolls}\label{S:PreRes}

In this section, we collect some preliminary results concerning scrolls  (cf.\ \cite{Seg}, \cite{GP1} and
\cite{GP2}).

Let $C$ be a smooth, projective curve of genus $g$ and let $\Ff$ be a rank-two vector bundle on $C$ as in
Assumptions \ref{ass:1}.

If $K_F$ denotes a canonical divisor of $F$, one has  that
\begin{equation}\label{eq:K}
K_F \equiv  -2 H + (d+2g-2)f,
\end{equation}where $\equiv$ is the numerical equivalence on ${\rm Div}(F)$ (see
e.g.\ \cite{Ha}). From \eqref{eq:K}, Serre duality and Riemann-Roch theorem we have 
$$R+1:= h^0(\Oc_F(1)) = d - 2 g + 2 + h^1(\Oc_F(1)).$$

\begin{definition}\label{def:spec}
The non-negative integer $h^1 (\Oc_F(1))$ is called the {\em speciality} of the scroll and will be denoted by
$h^1(F)$, or simply by $h^1$, if there is no danger of confusion.
Thus,
\begin{equation}\label{eq:R}
R= d - 2 g + 1 + h^1,
\end{equation} and the pair $(\Ff, C)$ determines $S \subset \Pp^R$ as a
{\em linearly normal scroll of degree $d$, genus $g$ and
speciality $h^1$}. Such a scroll $S$ is said to be {\em special} if $h^1 > 0$,
{\em non-special} otherwise.
\end{definition}

This definition coincides with the classical one given by Segre in \cite[\S\;3, p.\ 128]{Seg}:
Segre denotes by $n$ the degree of the scroll, by $p$ the sectional genus and by $ i :=  g-h^1$.

Since $R \geq 3$, then $d \geq 2g+2-h^1$. In particular, for non-special scrolls
\begin{equation}\label{eq:dbound}
d \geq 2g + 2.
\end{equation}

The following lemma provides an upper-bound for the speciality (cf.\ \cite[\S\,14]{Seg},
\cite{Ghio} and \cite[Lemma 5.7]{CCFMLincei}).

\begin{lemma}\label{lem:segre}
Let $C$ be a smooth, projective curve of genus $g \geq 1$ and let $F =
\Pp(\Ff)$ be a ruled surface on $C$ as in Assumption \ref{ass:1}. If $\det(\Ff)$ is non-special, then
\begin{equation}\label{eq:segre}
h^1 \leq g.
\end{equation}

In addition, if $d \geq 2g +2$, the equality holds in \eqref{eq:segre} if and only if $\Ff =
\Oc_C \oplus L$, in which case $\Phi = \Phi_{|\Oc_F(1)|} $ maps $F$ to a cone $S$
over a projectively normal curve of degree $d$ and genus $g$ in
$\Pp^{d-g}$.
\end{lemma}
\begin{proof} The bound \eqref{eq:segre} follows from the exact sequence \eqref{eq:sfiga1}, corresponding to a smooth
element $H \in |\Oc_F(1)|$ (cf.\ \eqref{eq:Ciro410b}).

If the equality holds, then $h^0(\Oc_F(1)) = h^0(\Ff) = d - g + 2$. If $B \in |H|$
is the curve corresponding to the section of $\Ff$ given by \eqref{eq:sfiga1}, then
$\Phi(B)$ is a smooth curve of degree $d$ and genus $g$, which is linearly normal in $\Pp^{d-g}$. This curve is
projectively normal (cf.\ \cite{Cast}, \cite{Matt} and \cite{Mumf}). Therefore
$F$ is mapped via $\Phi$ to a surface $S$ which is projectively normal, since
its general hyperplane section is (cf.\ \cite[Theorem 4.27]{Greco}). In particular, $h^1(\Oc_S) = 0$.

Since $S$ is birational to a ruled surface of positive genus, then $S$ cannot be smooth, and it has some normal
(isolated) singularities. This
forces $S$ to be a cone (cf.\  \cite[Claim 4.4]{CLM2}). Hence,
the assertion follows.
\end{proof}

From \eqref{eq:R} and from Lemma \ref{lem:segre}, we have $$d - 2g + 1 \leq  R \leq d -g + 1,$$
where the upper-bound is realized by cones whereas the lower-bound by
non-special scrolls (cf.\ \cite[Theorem 5.4]{CCFMLincei}).
Any intermediate value of $h^1$ can be realized, e.g.\ by
using decomposable vector bundles, as the following
example shows (see \cite[pp.\ 144-145]{Seg} and \cite[Example 5.11]{CCFMLincei}).

\begin{example}\label{ex:contoflam}
Let $g \geq 3$, $d \geq 4g -1$ and $ 1 \leq h^1 \leq g-1$ be
integers. Let  $L$ be a line bundle on $C$, such that $|L|$ is
base-point-free and $h^1(L) = h^1$.

Let $D$ be a general divisor on $C$ of degree $d - \deg (L)$. Notice that,
since $\deg(L) \leq 2g-2$, then  $\deg(D) \geq
2g+1$, so the linear series $|D|$ is very ample.

Consider $\Ff = L \oplus \Oc_C(D)$. If $F = \Pp(\Ff)$ then
$|\Oc_F(1)|$ is base-point-free and $h^1(F) = h^1$.

For large values of $h^1$, $|\Oc_F(1)|$ is rarely very ample
(cf.\ the case $h^1 = g$ in Lemma \ref{lem:segre}). For example

\begin{itemize}
\item[(i)] if $h^1 = g-1$, then $|L|$ is a $g^1_2$ on $C$.
In this case, $S$ has a double line $\Gamma$ because $|\Oc_F(1)|$ restricts as the $g^1_2$
to the section corresponding to the surjection  $\Ff \to\!\!\!\to L$;
\item[(ii)] if $h^1 = g-2$, then either $C$ is hyperelliptic and $|L| = 2 g^1_2$, or $C$
is trigonal and $|L| = g^1_3$ or $g=3$ and $L = \omega_C$. In the
former case, as in $(i)$, $S$ contains a double conic $\Gamma$; in the second case, $S$ has a
triple line. Only in the latter case, $S$ may be smooth,
and contains a smooth, plane quartic as a section.
\end{itemize}The analysis of the interplay between the $h^1$
and the smoothness of the scroll is rather subtle in general,
and we do not dwell on this here. For other examples,
we refer the reader to \cite[pp.\ 144-145]{Seg}, and to \cite{APS}, \cite{GP3} and \cite{GiSo}.

\end{example}

We now want to discuss general properties of some unisecant
curves on scrolls.

\begin{definition}\label{def:lndirec} Let $\Gamma_1 \subset S$ be a unisecant curve of $S$ of degree
$m$ such that $(S,\Gamma_1)$ is
associated to
\begin{equation}\label{eq:Fund0}
0 \to N \to \Ff \to L \oplus \Oc_{\underline{a}}\to 0,
\end{equation}where $\underline{a} \in {\rm Div}(C)$, possibly $\underline{a} = 0$.
Denote by $\Gamma$ the unique section contained in $\Gamma_1$. We will say that:
\begin{itemize}
\item[(i)] $\Gamma_1$ is {\em special} if $h^1(C, L)>0;$
\item[(ii)] $\Gamma_1$ is {\em linearly normally embedded} if
$ H^0(\Ff) \twoheadrightarrow H^0(L \oplus \Oc_{\underline{a}}^{\oplus 2}).$
\end{itemize}If $\Gamma_1 = \Gamma$ then $(ii)$ is equivalent
to $\Gamma = \Phi_{|L|}(C).$
\end{definition}

By Definition \ref{def:spec} and Formulae \eqref{eq:R}, \eqref{eq:dbound} and \eqref{eq:Fund0}, we
immediately have:

\begin{proposition}\label{prop:oss1}
Let $S \subset \Pp^R $ be a linearly normal scroll of degree
$d \geq 2g+2$, genus $g\geq 1$ and speciality $h^1\geq 0$,
determined by the pair $(\Ff,C)$.

\noindent (i) Let $\Gamma_1 \subset S$ be a unisecant curve of $S$, which is linearly normally
embedded, and let $(S,\Gamma_1)$ be associated with \eqref{eq:Fund0}.
Then:
\begin{itemize}
\item if at least one of
the two line bundles $L$ and $N$ in \eqref{eq:Fund0}
is special, then $S$ is a special scroll;
\item if both the line bundles $L$ and $N$ are
non-special, then $S$ is non-special and $$R + 1 = h^0(\Ff) =
h^0(L) + h^0(N) + 2 a.$$
\end{itemize}

\noindent $(ii)$ Conversely, the speciality of any unisecant curve of $S$
is less than or equal to the speciality of $S$. In particular, if $S$ is non-special, then $S$ contains
only non-special unisecant curves.
\end{proposition}

Moreover, we have:

\begin{proposition}\label{prop:oss2}
Let $S \subset \Pp^R $ be a linearly normal, non-special scroll of
genus $g \geq 1$ and degree $d \geq 2g+2$. Then each unisecant curve of
$S$ of degree$$m \leq d - 2g +
1$$is linearly normally embedded in $S$.
\end{proposition}
\begin{proof} Assume by contradiction that $\Gamma$ is not
linearly normally embedded in $S$, i.e.\ the map $H^0(\Ff) \to H^0(L \oplus \Oc_C(\underline{a}))$ is not surjective.
Equivalently, $h^1(C,N) >0$. Hence, $|N| = g_{d-m}^{d-m-g + h^1(N)}$ is a special
linear series on $C$.

From Clifford's Theorem $$2(d-m-g + h^1(N))
\leq d-m,$$which gives$$m \geq d -2g+ 2h^1(N) \geq d-2g+2$$which is
contrary to the assumptions.
\end{proof}

The two results above are stated in \cite[\S\, 4, p.\ 130 and p.\ 137]{Seg} and proved apparently
in a rather intricate way.

\section{Results on rank-two vector bundles on curves}\label{S:VectBun}

We  recall  here some results on rank-two vector
bundles on curves which are frequently used in what follows. For
complete details, we refer the reader to e.g.\ \cite{New} and
\cite{Ses}.

Let $C$ be a smooth, projective curve of genus $g \geq 0$. Let $\Ef$ be a
vector bundle of rank $r \geq 1$ on $C$. The {\em slope} of $\Ef$,
denoted by $\mu(\Ef)$, is defined as
\begin{equation}\label{eq:slope}
\mu(\Ef) := \frac{\deg(\Ef)}{r}.
\end{equation}

\noindent From now on, we shall be interested in the rank-two
case.

A rank-two vector bundle $\Ff$ on $C$ is
said to be {\em indecomposable}, if it cannot be expressed as a
direct sum $L_1 \oplus L_2$, for some $L_i \in {\rm Pic}(C)$, $1 \leq i
\leq 2$, and {\em decomposable} otherwise.

Furthermore, $\Ff$ is said to be:
\begin{itemize}
\item {\em semistable}, if for any sub-line bundle  $N \subset \Ff
$, one has $\deg(N) \leq \mu(\Ff)$;
\item {\em stable}, if for any sub-line bundle $N \subset \Ff
$, one has $\deg(N) < \mu(\Ff)$;
\item {\em strictly semistable}, if it is semistable and
there is a sub-line bundle $N \subset \Ff
$ such that $\deg(N) = \mu(\Ff)$;
\item {\em unstable}, if there is a sub-line bundle $N \subset \Ff
$ such that $\deg(N) > \mu(\Ff)$. In this case, $N$ is called a {\em destabilizing} sub-line bundle of $\Ff$.
\end{itemize}

\noindent
Recall the following well-known fact:

\begin{lemma}\label{lem:23ott} In the above setting, if $\Ff = L_1 \oplus L_2$ is decomposable,
then it is unstable unless $\deg(L_1) = \deg(L_2)$, in which case $\Ff$ is strictly semistable.
\end{lemma}
\begin{proof} Set $k = \deg(L_1) = \deg(L_2)$. Let $N \subset \Ff$ be any
sub-line bundle; the injection $ N \hookrightarrow \Ff$ gives
rise to the injection $$\Oc_C \hookrightarrow \Ff  \otimes
N^{\vee} = (L_1 \otimes N^{\vee}) \oplus (L_2 \otimes
N^{\vee}),$$i.e.\ to a global section of $(L_1 \otimes N^{\vee})
\oplus (L_2 \otimes N^{\vee})$. Hence, there is $i \in \{1,2\}$ such that $L_i \otimes N^{\vee}$ is effective,
i.e.\ $\deg(N) \leq \deg(L_i) = k = \mu(\Ff)$, which shows the semistability of $\Ff$.
\end{proof}

It is well-known that, given an integer $d$, there exists the moduli space of rank-two,
semistable vector bundles of degree $d$ on $C$, which we denote by $U_C(d)$. This is a projective variety, and
we denote by $U_C^s(d) \subseteq U_C(d)$ the open subscheme parametrizing
stable bundles.

When $g=0$, every vector bundle of rank
higher than one is decomposable (cf.\ e.g.\ \cite[Thm.\ 2.1.1]{OSS}). Furthermore, there is no stable
vector bundle of rank $r >1$ on $\Pp^1$ (see e.g.\ \cite[Corollary
5.2.1]{New}). In particular, $U_{\Pp^1}^s(d) = \emptyset$ for
any $d$.

When $g =1$,  we have to distinguish two cases. If $d$ is odd, $U_C(d) =
U_C^s(d) \cong C$ and every $[\Ff] \in U_C(d)$ is indecomposable.
If $d$ is even, $U_C^s(d) = \emptyset$ (cf.\ \cite[Rem.\ 5.9]{New}),
$U_C(d) \cong {\rm Sym}^2(C)$ and every $[\Ff] \in
U_C(d)$ is the direct sum of two sub-line bundles, each of degree $d/2$. For
details, see \cite[Theorem 16]{Tu} and the proof of Theorem \ref{thm:scrollfib1}-(i) later on.

For $g \geq 2$,  we have the following picture:

\noindent (i) $U_C(d)$ is irreducible, normal, of dimension $4g-3$;

\noindent (ii) $U_C^s(d)$ coincides with the set of smooth points of $U_C(d)$.

More precisely (cf.\ \cite[\S\,5]{New}):
\begin{itemize}
\item[(1)] if $d$ is odd, then $U_C(d) = U_C^s(d)$ is smooth,
i.e.\ each semistable vector bundle is stable;
\item[(2)] if $d$ is even, there are strictly semistable vector bundles, i.e.\ the inclusion
$U_C^s(d) \subset U_C(d)$ is proper.
\end{itemize}

\begin{proposition}\label{prop:sstabh1}
Let $C$ be a smooth, projective curve of genus $g \geq 1$ and let
$d$ be a positive integer.

\begin{itemize}
\item[(i)] If $d \geq 4g-3$ then,  for any
$[\Ff] \in U_C(d)$, $h^1(C, \Ff) = 0$;
\item[(ii)] if $g \geq 2$ and $d \geq 2g$ then, for $[\Ff] \in U_C(d)$ general,
$h^1(C,\Ff) = 0$.
\end{itemize}
\end{proposition}
\begin{proof} (i) For the proof see \cite[Lemma 5.2]{New}.

\noindent
(ii) We use a degeneration argument. Two cases must be
considered.

\noindent (a) Assume  $d = 2k$, with $k \geq g$. Let
$L_1, L_2 \in {\rm Pic}^k(C)$ be general line bundles. Let
$$\Ff_0 := L_1 \oplus L_2.$$Since $h^1(L_i) = 0$, for $1 \leq i
\leq 2$, then $h^1(\Ff_0) = 0$. Observe that $\mu(\Ff_0) = k$. By Lemma \ref{lem:23ott}, $\Ff_0$ is semistable,
i.e.\ $[\Ff_0] \in U_C(2k)$. By semicontinuity, $h^1(\Ff)= 0$ for general $[\Ff] \in U_C(2k)$.

\noindent (b) Assume $d = 2k+1$, so $k
\geq g$. Let $L \in {\rm Pic}^{k+1}(C)$ and $N \in {\rm Pic}^k(C)$ be general line bundles.
Let $\Ff_0$ be a general rank-two vector
bundle fitting in the exact sequence
\begin{equation}\label{eq:4/10b}
0 \to N \to \Ff_0 \to L \to 0.
\end{equation}As before, since $h^1(L) = h^1(N) =0$, then $h^1(\Ff_0) = 0$.
Furthermore, $\mu(\Ff_0) = k + \frac{1}{2}$. We want to show that $\Ff_0
$ is stable.

Since $\Ff_0$ corresponds to the general element in
${\rm Ext}^1(L,N)$, then the sequence \eqref{eq:4/10b} is unsplit,
since ${\rm dim}({\rm Ext}^1(L,N)) = h^1(N \otimes L^{\vee}) = g$.

Let $T \subset \Ff_0$ be any sub-line bundle. We have the
following commutative diagram:
\[
\xymatrix@C-2mm@R-2mm{ & & 0 \ar[d] \\
& & T \ar[d] \ar[dr]^{\varphi} \\
0 \ar[r] & N \ar[r] & \Ff_0 \ar[r] & L \ar[r] & 0.
}
\]

\noindent (i) If $\varphi$ is the zero-map, then $T$ is a sub-line
bundle of $N$, so $\deg(T) \leq k < \mu(\Ff_0)$.

\noindent (ii) If $\varphi$ is not the zero-map, then $\varphi$
is  injective, hence $L \otimes T^{\vee}$ is
effective. Thus, $\deg(T) \leq k+1$ and the equality holds if and
only if $L \cong T$. In the latter case, the exact sequence$$0 \to
N \to \Ff_0 \to L \to 0$$would split, against the assumption.

Therefore, $\deg(T) < \mu(\Ff_0)$, i.e.\ $[\Ff_0] \in U_C^s(2k+1)$.
One concludes the arguments using semicontinuity.
\end{proof}

\section{Hilbert schemes of non-special scrolls}\label{S:HNSS}

From now on, we shall focus on linearly normal, non-special
scrolls $S$ of degree $d$ and genus $g$. Therefore, from \eqref{eq:R},
$S \subset \Pp^R$ where
\begin{equation}\label{eq:R0}
R = d-2g + 1,
\end{equation}with $d$ as in \eqref{eq:dbound}.

If we moreover assume that $S$ is smooth,
of genus $g \geq 1$, one can deduce further restrictions on
$d$. Indeed, one has linearly normal, non-special
smooth scrolls of genus $g \geq 1$ only for
\begin{equation}\label{eq:deg0}
d \geq 5, \;\; {\rm when}\;\; g=1, \;\; {\rm and}\;\; d \geq 2g+4,
\;\; {\rm when}\;\; g\geq 2,
\end{equation}(cf.\ \cite[Remark 4.20]{CCFMLincei}).

Basic information about the Hilbert scheme of these scrolls are essentially contained 
in \cite{APS}. We recall the main results. First of all the following theorem 
(see \cite{APS}, for a more explicit statement cf.\ Theorem 1.2 in \cite{CCFMLincei}). 

\begin{theorem} \label{thm:Lincei}
Let $g \geq 0$ be an integer and set $k := {\rm min}\{1, g-1\}$. If
$ d \geq 2 g + 3 + k$, then there exists a unique, irreducible component $\HH_{d,g}$ of the
Hilbert scheme of scrolls of degree $d$, sectional genus $g$ in
$\Pp^R$ such that the general point $[S] \in \HH_{d,g}$ represents a
smooth, non-special and linearly normal scroll $S$.

\noindent
Furthermore,
\begin{itemize}
\item[(i)] $\HH_{d,g}$ is generically reduced;
\item[(ii)] ${\rm dim} (\HH_{d,g}) = 7(g-1) + (d-2g+2)^2 = 7(g-1) + (R+1)^2$;
\item[(iii)] $\HH_{d,g}$ dominates the moduli space
${\mathcal M}_g$ of smooth curves of genus $g$.
\end{itemize}
\end{theorem}

For $g \geq 1$,  the next result gives more information about the general scroll
parametrized by $\HH_{d,g}$ (cf. \cite{APS}; we give here a short, independent proof). 

\begin{theorem}\label{thm:scrollfib1} Let $g \geq 1$ and  let
$d$ and $R$ be as in Theorem \ref{thm:Lincei}.
Let $[S] \in \HH_{d,g}$ be a general point and $(\Ff, C)$ be a pair which determines $S$, where
$[C ] \in {\mathcal M}_g$ general. Then $[\Ff]$ is a general point  in $U_C(d)$.

\noindent
Denote by $G_S \subset PGL(R+1,\C)$ the subgroup of projective transformations fixing $S$.

\noindent
(i) If $g = 1$, then
\begin{itemize}
\item when $d$ is odd, then  $G_S = \{1\}$ and the
pair $(\Ff, C)$ determining $S$ is unique.

\item when $d$ is even, then ${\rm dim}(G_S) = 1$, and accordingly there is a $1$-dimensional family
of vector bundles $[\Ff] \in U_c(d)$, such that the pairs $(\Ff,C)$ determine $S$.
\end{itemize}

\noindent
If $g \geq 2$, then $G_S = \{1\}$ and the
pair $(\Ff, C)$ determining $S$ is unique.

\end{theorem}
\begin{proof} We first consider the case $g \geq 2$. Denote by ${\mathcal U}_{d} \stackrel{\tau}{\to} {\mathcal
M}_g$ the relative moduli stack of degree $d$, semistable, rank-two vector bundles so that, for $[C] \in
{\mathcal M}_g$, $\tau^{-1}([C]) = U_C(d)$. Since, for any $[C] \in {\mathcal M}_g$,
$U_C(d)$ is irreducible of dimension $4g-3$, then ${\mathcal U}_{d}$ is
irreducible, of dimension $7g-6$.

Let ${\Ff}_{\mathcal U} \stackrel{\pi}{\longrightarrow} {\mathcal U}_{d}$ be the universal
bundle. From Proposition \ref{prop:sstabh1} - (ii), on an open, dense subscheme
$U \subseteq {\mathcal U}_{d}$, $\pi_*({\Ff}_{\mathcal U})|_U$ is a vector bundle of rank $R+1$.
After possibly shrinking $U$, we may assume that this vector bundle is trivial on $U$
and we can choose indipendent global sections $s_0, \ldots, s_R$ of
$\pi_*({\Ff}_{\mathcal U})|_U$.

Consider ${\mathcal P}_{\mathcal U} :=  U \times PGL(R+1, \C)$ which is irreducible,
of dimension $7g-7 + (R+1)^2$. An element of ${\mathcal P}_{\mathcal U}$ can be regarded as
a triple $(C, \Ff, \rho)$, where $[C] \in \M_g$ is a general curve, $[\Ff] \in U_C(d)$ is general,
and $\rho$ is a projective transformation. Moreover, the sections $s_0, \ldots, s_R$ induce indipendent
sections of $H^0(C, \Ff)$ and therefore determine a morphism $F = \Pp(\Ff) \to S \subset \Pp^R$.

Let ${\rm Hilb}(d,g)$ denote the Hilbert scheme of scrolls of degree $d$
and genus $g$ in $\Pp^{R}$. Consider  the morphism
$$\Psi: {\mathcal P}_{\mathcal U}\to
{\rm Hilb}(d,g)$$which maps the triple $(C, \Ff, \rho)$ to the surface $\rho(S)$.

Let ${\mathcal S}_U $ be the image of $\Psi$, which is irreducible. From
Proposition \ref{prop:sstabh1} and Theorem \ref{thm:Lincei}, ${\mathcal S}_U $ is
contained in $\HH_{d,g}$. Since semistability is an open condition
(cf.\ e.g.\ \cite[Proposition 2.3.1]{HL}), then $\Psi$ is dominant onto $\HH_{d,g}$.

Recalling the description of $U_C(d)$ in \S \,\ref{S:VectBun}, it follows that the
general scroll parametrized by a point in $\HH_{d,g}$  corresponds to a stable vector bundle.

From Theorem \ref{thm:Lincei}, we deduce that $\Psi$ is generically finite onto the image. More
precisely, let $(C, \Ff, \rho)$ and $(C, \Ff', \rho')$ be points in ${\mathcal P}_{\mathcal U}$ such
that $\rho(S) = \Psi(C, \Ff, \rho) = \Psi(C, \Ff', \rho') = \rho(S')$.
Hence $\phi = (\rho')^{-1} \circ \rho : S \to S' $ is a projective transformation. Let
$F = \Pp(\Ff) \cong \Pp(\Ff') = F'$. Since $C$, being general,
has no non-trivial automorphisms, $\phi$ induces an automorphism $\Phi : F \to F'$ and
$\Phi^*(\Oc_{F'}(1)) = \Oc_F(1)$.

Since $\Pp(\Ff) \cong \Pp(\Ff')$, there is a line bundle $\eta \in {\rm Pic}^0(C)$ such that
$\Ff' \cong \Ff \otimes \eta$. Note that ${\rm Pic}^0(C) \subset {\rm Pic}^0 (F) = {\rm Pic}^0 (F')$. Thus,
$\Oc_{F}(1) \cong \Phi^* (\Oc_{F'}(1)) \otimes \eta$. This implies that $\eta$ is trivial.

Finally, we claim that $\rho = \rho'$, i.e.\ $\phi = {\rm id}$. Otherwise, $\Phi : F \to F$
would  come from a non-trivial automorphism of $\Ff$ which is not possible because
$\Ff$, being stable, is simple (cf.\ e.g.\ \cite[Corollary 5.1.1]{New}).

\vskip 5pt

\noindent
Now let $g=1$. As above, we can
consider ${\mathcal U}_{d} \stackrel{\tau}{\to} {\mathcal M}_1$, the open subset
$U \subset {\mathcal U}_{d}$  and the variety ${\mathcal P}_{\mathcal U} = U \times PGL(R +1,\C)$.  Then, ${\mathcal P}_{\mathcal U}$ is irreducible and \[{\rm dim}({\mathcal P}_{\mathcal U}) = \left\{ \begin{array}{cl}
              d^2 + 1, & {\rm if} \; d \; {\rm is \; odd},\\
              d^2 + 2, & {\rm if} \; d \; {\rm is \; even},
           \end{array}
  \right.
\] (cf.\ \S\,\ref{S:VectBun}).
For the same reasons as above, the map
$\Psi: {\mathcal P}_{\mathcal U}\to \HH_{d,1}$ is dominant.

If $[S] \in \HH_{d,1}$ denotes the general point and
if ${\mathcal P}_S := \Psi^{-1}([S])$, we have:
\[{\rm dim}({\mathcal P}_S) = \left\{ \begin{array}{cl}
               1, & {\rm if} \; d \; {\rm is \; odd},\\
               2, & {\rm if} \; d \; {\rm is \; even}.
           \end{array}
  \right.
\]

Indeed, assume that $S$ is determined by a pair $(\Ff, C)$.
Let $h = g.c.d.(2,d)$, $2=hn'$ and $d = hd'$. Then, as in \cite[Theorem 16]{Tu},
any $[\Ff] \in U_C(d)$ is of the form
$$\Ff = E(n',d') \otimes (\bigoplus_{i=1}^h L_i),$$
where $E(n',d')$ is the {\em Atiyah's bundle} of rank $n'$ and degree $d'$ and where
the $L_i$'s are line bundles of degree $0$, determined up to
multiplication by a $n'$-torsion element in $J(C) \cong C$, $1
\leq i \leq h$.

If $d$ is odd, then $h=1$, $n' = 2$, $d = d'$, so any $[\Ff] \in
U_C(d)$ is of the form $\Ff = E(2,d) \otimes L$, for some $L \in J(C)$.
Since $C$ is an elliptic curve, for any $x \in C$:
\[
\begin{array}{rrcll}
t_x : & C & \stackrel{\cong}{\longrightarrow} & C \\
    & p & \longrightarrow & p + x,
\end{array}
\]is an automorphism of $C$. Since $U_C(d) \cong C$, this defines a natural action of $C$ on
$U_C(d)$: fix $[\Ff] \in U_C(d)$, then for every $M \in J(C)$,
$$t_{[M^{\vee}]}^*(\Ff)= \Ff \otimes M  \in U_C(d).$$From
\cite{LaTu}, $t_{[M^{\vee}]} \in {\rm Aut} (C)$ lifts to an automorphism
of the elliptic ruled surface $\Pp(\Ff)$. As a consequence, for any
other $[\Ff'] \in U_C(d)$, $\Pp(\Ff) \cong \Pp(\Ff')$.
Since ${\rm dim}({\mathcal P}_S) = 1$, an argument completely
similar to the one we made in the case $g \geq 2$ proves that $G_S = \{1\}$.

If $d$ is even, then $h=2$, $n' = 1$ and $d' = d/2$, so any $[\Ff]
\in U_C(d)$ is of the form  $\Ff = E(1,d/2) \otimes (L_1 \oplus L_2)$, for
some $L_1, \; L_2 \in J(C)$. In the same way, any other
$[\Ff'] \in U_C(d)$ of the form $\Ff' = E(1,d/2) \otimes ((L_1
\otimes N) \oplus (L_2 \otimes N))$, for some $N\in J(C)$, and $\Pp(\Ff) \cong \Pp(\Ff')$.
Since ${\rm dim}({\mathcal P}_S) = 2$, this implies that ${\rm dim}(G_S) = 1$.
\end{proof}

\begin{remark}\label{rem:deven}
If $g=1$ and $d$ is even, the elements of the one-dimensional group $G_S$ correspond
to non-trivial endomorphisms of the corresponding strictly semistable vector bundle $\Ff$.
This can be read off by looking at the projective geometry of $S$.
Indeed, $S$ contains sections $\Gamma_i$ of degree $d/2$ associated with the line bundles
$M_i = E(1,d/2) \otimes L_i$, $1 \leq i \leq 2$. These two sections are disjoint,
since $\Ff$ is decomposable (cf.\ \cite[p.\ 383]{Ha}), and are
the curves of minimal degree of $S$, since they are determined by
quotients via sub-line bundles of maximal degree (cf.\ also Theorem \ref{thm:ganzo}).

Let $\Lambda_i := \langle\Gamma_i \rangle \cong
\Pp^{\frac{d-2}{2}}$, for $ 1 \leq i \leq 2$. Then $\Lambda_1 \cap
\Lambda_2 = \emptyset$ so, if $G$ denotes the connected subgroup of
$PGL(R+1, \C)$ of elements which pointwise fix these two skew
linear subspaces of $\Pp^R$, then ${\rm dim} (G) = 1$ and each element
of $G$ fixes $S$. This shows that $G$ is the connected component of the identity
of $G_S$.
\end{remark}

\begin{remark}\label{rem:hilbparam1} In \cite{CCFMLincei}, we gave an explicit dimension count
for $\HH_{d,g}$ (cf.\ \cite[Theorem 5.4 and Remark 5.6]{CCFMLincei}).
Theorem \ref{thm:scrollfib1} gives another way of making the same computation.

Precisely, when $g \geq 2$, the number of parameters on which the general point of
$\HH_{d,g}$ depends, is given by the following count:
\begin{itemize}
\item $3g-3$ parameters for the class of the curve $C$ in ${\mathcal M}_g$, plus
\item $4g-3$ parameters for the general rank-two vector bundle in $U_C(d)$, plus
\item $(R+1)^2 -1$ parameters for projective transformations in $\Pp^{R}$.
\end{itemize}When $g = 1$, we have:
\begin{itemize}
\item $1$ parameter for $C$ in ${\mathcal M}_1$, plus
\item $1$ or $2$ parameters (according to the cases $d$ odd or $d$ even),
for the general rank-two vector bundle in $U_C(d)$, plus
\item $(R+1)^2 -1 - {\rm dim}(G_S)$ parameters for projective transformations in $\Pp^{R}$ (with
${\rm dim}(G_S) = 0$ or $1$, according to the cases $d$ odd or $d$ even), minus
\item $1$ parameter, for the $C$-action on $U_C(d)$.
\end{itemize}
\end{remark}

\begin{remark}\label{rem:hilbparam2} Even if $d$ is
large with respect to $g$ (cf.\ Proposition \ref{prop:sstabh1} - (i)),
Theorem \ref{thm:scrollfib1} does not imply
that all smooth scrolls in $\HH_{d,g}$
come from either a stable or a semistable rank-two vector bundle
on $C$. Indeed, let $\Ff$ be any rank-two vector bundle on a curve $C$ of genus $g$.
By twisting $\Ff$ with a sufficiently high multiple of an ample line bundle $A$ on $C$, we have
a new vector bundle $\Ff' = \Ff \otimes A^{\otimes k}$, such that $h^1(C, \Ff') = 0$ and
$\Oc_{\Pp(\Ff')} (1) $ is very ample. By embedding $\Pp(\Ff')$ via $|\Oc_{\Pp(\Ff')} (1)|$,
one has a smooth, non-special, linearly normal scroll $S$ of a certain degree $d$,
and therefore $[S] \in  \HH_{d,g}$.

More precisely, look at the following example. Let $C$ be any smooth, projective curve of genus $g \geq 2$ and
$k \geq 2$ be an integer. Let $L \in {\rm Pic}^{2g+ k}(C)$ and $N
\in {\rm Pic}^{2g+ k -1}(C)$ be general line bundles. Let
$\Ff_0$ be a general, rank-two vector bundle on $C$ fitting in the
exact sequence $$ 0 \to L \to \Ff_0 \to N \to 0.$$Thus,
$\deg(\Ff_0) = 4g+2k-1$ and, by the generality assumption on $L$ and
$N$, $h^1(C, \Ff_0) = 0$. Furthermore, by degree assumptions, both
$L$ and $N$ are very ample on $C$. Therefore, the pair $(\Ff_0,
C)$ determines a smooth scroll $S$ which is non-special and
linearly normal in $\Pp^{2g+2k}$, i.e.\ $[S] \in \HH_{4g+2k -1,
g}$.

Nonetheless, $\Ff_0$ is unstable on $C$: indeed, $\mu(\Ff_0) =
2g+k - \frac{1}{2}$ whereas $\deg(L) = 2g + k > \mu(\Ff_0)$, so $L$
is a destabilizing sub-line bundle of $\Ff_0$.

In accordance with Theorem \ref{thm:scrollfib1}, the reader can verify that
the number of parameters on which scrolls of this type depend
is at most$$ 6g-5 + (2g+2k+1)^2 < {\rm dim}(\HH_{4g+2k -1, g}).$$
\end{remark}

\begin{remark}\label{rem:CxP1} In \cite[Theorem 1.2]{CCFMLincei} we proved that there are points in $\HH_{d,g}$
corresponding to unions of planes with Zappatic singularities. This means that smooth
surfaces  in $\HH_{d,g}$ degenerate to these unions of planes. It is interesting to remark that this applies
to surfaces of the type $C \times \Pp^1$, suitably embedded as scrolls corresponding to points in $\HH_{d,g}$.
First of all, let $C$ be any curve of genus $g$ and let $L$ be a very-ample non-special line bundle of degree
$d \geq g +3$. The global sections of $L$ determine an embedding of $C$ in $\Pp^{d-g}$. Consider
the Segre embedding of $C\times \Pp^1$. This gives a linearly normal, non-special smooth
scroll of degree $2d$ in $\Pp^{2d-2g+1}$ and the corresponding point sits in $\HH_{2d,g}$. Moreover,
by looking at the argument in \cite[\S\,3]{CCFMLincei}, one sees that the planar Zappatic surface
$X_{2d,g}$ contained in $\HH_{2d,g}$ is a limit of a product.
In Zappa's original paper \cite{Zapp}, the author remarked that products $C \times \Pp^1$
can be degenerated to union of quadrics, leaving as an open problem to prove the degeneration to union of planes.

\end{remark}

We finish this section by constructing suitable reducible surfaces, frequently
used in the rest of the paper, corresponding to points in $\HH_{d,g}$.
The first construction is contained in  \cite[Constructions 4.1, 4.2, Claim 4.3,
Theorem 4.6, Lemma 4.7]{CCFMLincei} (cf.\ also \cite{CLM}), and will be stated below for the reader's convenience.

\begin{constr}\label{const:T} {\em Let $g \geq 1$, $d$ and $\HH_{d,g}$ be as in Theorem \ref{thm:Lincei}. Then
$\HH_{d,g}$ contains points $[T]$ such that $T$ is a reduced, connected, reducible surface,
with global normal crossings, of the form
\begin{equation}\label{eq:T}
T:= X \cup Q,
\end{equation}where $X$ is a scroll corresponding to a general point of $\HH_{d-2,g-1}$ and
$Q$ is a smooth quadric, such that $X \cap Q = l_{1} \cup l_{2}$,
where  $l_{i}$ are general rulings of $X$, for $1 \leq i \leq 2$,
and the intersection is transverse.

Furthermore, if $\N_{T/\Pp^R}$ denotes the normal sheaf of $T$ in $\Pp^R$, then
$h^1(T, \N_{T/\Pp^R}) = 0$; in particular, $[T]$ is a smooth point of $\HH_{d,g}$.}
\end{constr}

The second construction is similar. Precisely, we have:

\begin{constr}\label{const:Y} {\em Let $g \geq 1$, $d$ and $\HH_{d,g}$ be as in Theorem \ref{thm:Lincei}. Then
$\HH_{d,g}$ contains points  $[Y]$ such that $Y$ is a reduced, connected, reducible surface, with global normal
crossings, of the form
\begin{equation}\label{eq:Y}
Y:= W \cup Q_1 \cup \cdots \cup Q_g,
\end{equation}where $W$ is a rational normal scroll, corresponding to a general point of $\HH_{d-2g,0}$, and
each $Q_j$ is a smooth quadric, such that
\begin{equation}\label{eq:aiuto23}
Q_j \cap Q_k = \emptyset, \;\; {\rm if} \;\; 1 \leq j \neq k \leq g
\end{equation}and
$$W \cap Q_j = l_{1,j} \cup l_{2,j},$$where  $l_{i,j}$ are general rulings of $W$, for $1 \leq i \leq 2$,
$1 \leq j \leq g$, and the intersection is transverse.

Furthermore, for any such $Y$, one has $h^1(Y, \N_{Y/\Pp^R}) = 0$; in particular, $[Y]$ is a smooth point of $\HH_{d,g}$.}
\end{constr}
\begin{proof}
Let $[W]\in \HH_{d-2g,0}$ be a general point. This corresponds to a smooth, rational normal scroll
of degree $d-2g$ in $\Pp^R$. Let $l_{1,j}, \, l_{2,j}$, $1 \leq j \leq g$, be $2g$ general
rulings of $W$. Let $\Pi_j$ be the $\Pp^3$ spanned by $l_{1,j}$ and  $l_{2,j}$. Let
$Q_j \subset \Pi_j $ be a general quadric, containing $l_{1,j}, \, l_{2,j}$, for $1 \leq j \leq g$.
Then, $Q_j$ is smooth and, for any $1 \leq j \leq g$,
$$W \cap Q_j = l_{1,j} \cup l_{2,j}$$and the intersection is transverse (for $g =1$, we have only
one quadric and the assertion follows, whereas for $g \geq 2$, see  \cite[Construction 4.2]{CCFMLincei}).

By generality and since $d \geq 2g + 4$, when $g \geq 2$, one sees that
${\rm dim} (\Pi_j \cap \Pi_k ) \leq 1$, for $ 1 \leq j \neq k \leq g$, and therefore we can assume
\eqref{eq:aiuto23}.

\noindent
Let $Y$ be as in \eqref{eq:Y}. Then $Y$ is of degree $d$, its sectional (arithmetic)
genus is $g$ and $h^1(Y, \Oc_Y(1)) = 0$ since it is clearly linearly normal in $\Pp^R$.
From Theorem \ref{thm:Lincei}, it follows that $[Y] \in \HH_{d,g}$.
Furthermore, as in the proof of  \cite[Theorem 4.6 and Lemma 4.7]{CCFMLincei}, $h^1(\N_{Y/\Pp^R}) = 0$
so $[Y]$ is a smooth point of $\HH_{d,g}$. Thus, $[Y]$ deforms to a general point $[S] \in \HH_{d,g}$.
\end{proof}

\section{Properties of the scheme of unisecant curves}\label{S:unisec}

In this section we prove Theorem \ref{thm:ganzo2} below, which  contains basic information on
the scheme parametrizing unisecant curves
of given degree $m$ on the scroll $S$, where $[S]$ is a general point  in $\HH_{d,g}$, with $g\geq 0$ and $d$ as in
Theorem \ref{thm:Lincei}.

First,  we recall some results in \cite{Ghio},
which are inspired by the work of C.\ Segre \cite[\S\,11, p.\ 138]{Seg}.

\begin{definition}[{see \cite[Definition 6.1]{Ghio}}] \label{def:ghio0}
Let $C$ be a smooth, projective curve of genus $g
\geq 0$. Let $F = \Pp(\Ff)$ be a geometrically ruled surface over
$C$ and let $d = \deg(\Ff)$. For any positive integer $m$, denote
by
\begin{equation}
{\rm Div}_F^{1,m}
\end{equation}the Hilbert scheme of unisecant curves of $F$, which are of degree $m$ with respect to $\Oc_F(1)$
(cf.\ Definition \ref{def:unisec}).
\end{definition}

\begin{remark}\label{rem:ghio}
By recalling \eqref{eq:Fund} and \eqref{eq:Fund2},
the elements in ${\rm Div}_F^{1,m}$ correspond to quotients of $\Ff$.
Therefore, ${\rm Div}_F^{1,m}$ has a natural structure as a Quot-scheme (cf.\ \cite{Groth}).
As such, ${\rm Div}_F^{1,m}$ has an {\em expected dimension}
\begin{equation}\label{eq:div1m}
d_m := {\rm max} \{ -1, \; 2 m - d - g + 1\}
\end{equation}
and therefore
\begin{equation}\label{eq:div1mb}
{\rm dim}({\rm Div}_F^{1,m}) \geq d_m.
\end{equation}
\end{remark}

\begin{definition}[{see \cite[Definition 6.1]{Ghio}}] \label{def:ghio}
Notation as in Definition \ref{def:ghio0}. $F$ is said to be a {\em general ruled surface} if:
\begin{itemize}
\item[(i)] $ {\rm dim}({\rm Div}_F^{1,m}) = d_m$;
\item[(ii)] ${\rm Div}_F^{1,m}$ is smooth, for any $m$ such that $d_m \geq 0$;
\item[(iv)] ${\rm Div}_F^{1,m}$ is irreducible, for any $m$ such that $d_m > 0$.
\end{itemize}
\end{definition}

\noindent
Note that being general is an open condition in $\HH_{d,g}$.

In \cite{Ghio}, there is an asymptotic existence result for general ruled surfaces.

\begin{theorem}[{cf.\ \cite[Th\'eor\`eme 7.1]{Ghio}}] \label{thm:Ghio}
Let $C$ be a smooth, projective curve of genus $g \geq 0$. There is a positive integer $\delta_C$ such that,
for every $d \geq \delta_C$, there is a general ruled surface of degree $d$ over $C$.
\end{theorem}

As a consequence, using \cite[Proposition 5.2 and Theorem 5.4]{CCFMLincei} and the proof of \cite[Th\'eor\`eme 7.1]{Ghio}, one has:

\begin{corollary}\label{cor:Ghio} For every $g \geq 0$, there is a positive integer $\delta_g$ such that,
for any $d \geq \delta_g$, the general point of $\HH_{d,g}$ corresponds to a general ruled surface.
\end{corollary}

Note that Ghione's argument gives no information about $\delta_g$.
Using a degeneration argument, it is possible to improve Corollary \ref{cor:Ghio} by specifing a bound on $\delta_g$.

\begin{theorem}\label{thm:ganzo2} Let $g$ and $d$ be as in Theorem \ref{thm:Lincei}.
If $[S] \in \HH_{d,g}$ is a general point, then $S$ is a general ruled surface.
\end{theorem}

\begin{proof} We proceed by induction on $g$. The case $g = 0$ is clear.

Assume $g>0$ and let either $d \geq 5$, if $g =1$, or $d \geq 2g+4$, if $g \geq 2$.
Let $[X] \in \HH_{d-2,g-1}$ be a general point. By induction, $X$ is a smooth, general ruled surface.
Let $l_1$ and $l_2$ be two general rulings of $X$.
Let $Q$ be a general quadric surface containing $l_1$ and $l_2$; thus $Q$ is smooth and $X$
and $Q$ meet transversally along $X \cap Q = l_1 \cup l_2$.

In particular, the surface $T:= X \cup Q$ is such as in Construction \ref{const:T}, so  $[T]$ is
a smooth point of ${\HH}_{d,g}$.

We consider a {\em section} $\Gamma_T$ of $T$  as a connected union
\begin{equation}\label{eq:uniT}
\Gamma_T = \Gamma_X \cup \Gamma_Q,
\end{equation}where $\Gamma_X$ (resp.\ $\Gamma_Q$) is a section of $X$ (resp.\ of $Q$),
such that $\Gamma_X$ and $\Gamma_Q$ meet transversally at $\Gamma_X \cap \Gamma_Q = \{p_1, p_2 \}$, where
$p_i \in l_i$, $1 \leq i \leq 2$. From Definition \ref{def:direct}, since $X$ and $Q$ are smooth then
$\Gamma_X$ and $\Gamma_Q$ are both smooth and irreducible,  hence $\Gamma_T$ is a reduced, reducible curve having
two nodes as its only singularities.

Similarly as in Definition \ref{def:ghio}, we will denote by
\begin{equation}\label{eq:DT}
{\rm Div}_T^{1,m}
\end{equation}the Hilbert scheme of curves on $T$ of degree $m$, arithmetic genus
$g$, which intersect at only one point the general line of the ruling of $X$ and
of the ruling $|l_1| = |l_2|$ of $Q$.

Thus, $[\Gamma_T] \in {\rm Div}_T^{1,m} $ is the union of two curves
\begin{itemize}
\item[(i)] $[\Gamma_X] \in {\rm Div}^{1,m_X}_X$, with $m_X < m$, and
\item[(ii)] $[\Gamma_Q] \in {\rm Div}_Q^{1,m-m_X}$,
\end{itemize}which match along $l_1 \cup l_2$.

By induction and by \eqref{eq:div1m}, we have:

\begin{itemize}
\item if $d + g - 3$ is even, then $\frac{d+g-3}{2} \leq m_X \leq
m-1$,
\item if $d+g-3$ is odd, then $\frac{d+g-4}{2} \leq m_X \leq
m-1$,
\item in any case, by induction, $d_{m_X} = 2 m_X - d - g + 4 = {\rm dim}({\rm Div}^{1,m_X}_X)$.
\end{itemize}Moreover, the scheme ${\rm Div}^{1,m_X}_X$ is smooth and, in addition, it is
irreducible unless $d_{m_X} = 0$. This is equivalent $d + g -3$ odd and $m_X =\frac{d+g-4}{2}$, in which
case ${\rm Div}^{1,m_X}_X$ consists of  finitely many distinct curves on $X$ (in Theorem \ref{thm:ganzo} - (ii),
we shall prove that the number of these curves is $2^{g-1}$). The scheme ${\rm Div}^{1, m - m_X}_Q$ is
not empty, irreducible and ${\rm dim}({\rm Div}^{1, m - m_X}_Q) = 2(m-m_X) - 1 \geq 1$.

For any $m_X$ as above, let $G_1^{m_X} = {\rm Div}^{1, m_X}_X  $ and
$G_2^{m_X} = {\rm Div}^{1, m - m_X}_Q $. We have two natural rational maps
$$\phi_i^{m_X} \colon G_i^{m_X} {\dashrightarrow} \Pp^1 \times \Pp^1, \;\; 1
\leq i \leq 2,$$ given by intersecting the general curve in $G_i^{m_X}$
with $l_1$ and $l_2$. Consider
\begin{equation}\label{eq:gms}
G_T^{m_X}
\end{equation}the closure of the fibre product of the maps $\phi_i^{m_X}$, for $1 \leq i \leq 2$.

\begin{claim}\label{cl:cod} Every irreducible component of $G_T^{m_X} $ has dimension
$d_m= 2m-d-g+1 \geq0 $. Moreover, if $d_m > 0$, then $G_T^{m_X} $ is irreducible.
\end{claim}

\begin{proof}[Proof of Claim \ref{cl:cod}] For the first part, observe that
the family $G_2^{m_X} $ is the linear system of curves of type
$(1, m - m_X - 1)$ on $Q$ and  ${\rm dim} (G_2^{m_X}) = 2 (m-m_X) -1$.

Consider the projection$$ G_T^{m_X} \stackrel{\pi_1}{\longrightarrow} G_1^{m_X}. $$

\noindent (a) If $m - m_X >1$, then $\pi_1$ is surjective and its fibres  are projective spaces of dimension
$2 (m - m_X) - 3$; hence the assertion follows.

\noindent (b) If $m - m_X =1$, then $G_T^{m_X} $ is isomorphic to $(\phi_1^{m_X})^{-1} (\phi_2^{m_X}(G_2^{m_X}))$.
Note that $\phi_2^{m_X}(G_2^{m_X})$ is a smooth, rational curve of type $(1,1)$ on $\Pp^1 \times \Pp^1$.
Since $Q$ is general, this curve is general in its linear system.
This implies that ${\rm dim}( G_T^{m_X}) = {\rm dim} (G_1^{m_X}) - 1$, proving the first assertion.

The assertion about the irreducibility of $G_T^{m_X}$, when $d_m > 0$, is clear in case (a) above.
In case (b), by induction,  $G_1^{m_X}$ is smooth and irreducible, since $d_{m_X} =
d_m +1 > 0$. The proof of case (b) shows that $G_T^{m_X} $ is the pull-back via $\phi_1^{m_X}$ of a general curve
of type $(1,1)$ on $\Pp^1 \times \Pp^1$. This is irreducible by Bertini's theorem.
\end{proof}

The general element of a component of $G_T^{m_X} $ is a pair
$(\Gamma_X, \Gamma_Q)$ such that neither $\Gamma_X$ nor $\Gamma_Q$ contains either $l_1$ or $l_2$.

Let $${\mathcal U} := \bigcup_{\lfloor \frac{d+g-3}{2} \rfloor \leq j \leq m-1} G_T^j,$$
with $G_T^j$ as in \eqref{eq:gms}.  When $d_m = 0$, then ${\mathcal U}= G_T^{\frac{d+g-3}{2}}$;
when $d_m > 0$, the irreducible components of
$ {\mathcal U}$ coincide with the $G_T^j$'s (cf.\ Claim \ref{cl:cod}).

Note that there is a natural map
\begin{equation}\label{eq:psi}
{\mathcal U} \stackrel{\psi}{\longrightarrow} {\rm Div}_T^{1,m}
\end{equation}which is surjective and birational on any irreducible component of ${\mathcal U}$. Therefore, the irreducible
components of ${\rm Div}_T^{1,m}$ are images of the irreducible components of ${\mathcal U}$. In particular, by
Claim \ref{cl:cod},$${\rm dim} ({\rm Div}_T^{1,m}) = d_m.$$

\begin{claim}\label{cl:notempty} If $[S] \in \HH_{d,g}$ is a general point  and
if ${\rm Div}_S^{1,m} \neq \emptyset $, then $d_m \geq 0$, i.e.\ $2m - d - g + 1 \geq 0$, and
moreover ${\rm dim}({\rm Div}_S^{1,m} ) = d_m$.
\end{claim}
\begin{proof}[Proof of Claim \ref{cl:notempty}] Note that, when $S$ degenerates to $T$, the flat limit of
${\rm Div}_S^{1,m}$ is contained in ${\rm Div}_T^{1,m}$. If ${\rm Div}_S^{1,m} \neq \emptyset $, then
${\rm Div}_T^{1,m} \neq \emptyset $. Therefore, there is a $\lfloor \frac{d+g-3}{2} \rfloor \leq j \leq m-1$
such $G_T^j \neq \emptyset$. Then $d_m \geq 0$ follows from Claim \ref{cl:cod}. By semicontinuity, one has
$${\rm dim} ({\rm Div}_S^{1,m}) \leq {\rm dim} ({\rm Div}_T^{1,m} ) = d_m.$$Equality holds by
\eqref{eq:div1mb}.
\end{proof}

The next step is the following:

\begin{claim}\label{cl:conn}
If $d_m > 0$, ${\rm Div}_T^{1,m}$  is connected.
\end{claim}

\begin{proof}[Proof of Claim \ref{cl:conn}] Let ${\rm Div}_T^{1,m}(j)$ be the irreducible component of
${\rm Div}_T^{1,m}$ which is the image of $ G_T^j$ via $\psi$. We will prove that, for every $j$ such that
$\lfloor \frac{d+g-3}{2} \rfloor \leq j \leq m-2$, the subschemes $ {\rm Div}_T^{1,m}(j)$ and
${\rm Div}_T^{1,m}(j+1)$ intersect.

Consider a general point of ${\rm Div}_T^{1,m}(j)$. This consists of the
union of two general curves $[\Gamma_1] \in {\rm Div}_X^{1,j}$,
$[\Gamma_2] \in {\rm Div}_Q^{1,m-j}$ with $\Gamma_1 \cap \Gamma_2 = \{p_1, p_2\}$,
where $p_i \in l_i$, $1 \leq i \leq 2$.

Since $m - j \geq 2$, the linear system $|\Gamma_2|$ has dimension at least three. Hence,
$${\rm dim}(|\Ii_{\{p_1, p_2\}/Q} \otimes \Oc_Q (\Gamma_2)|) \geq 1$$and therefore $\Gamma_2 $ degenerates inside the latter
linear system to $\overline{\Gamma}_2 + l_1$, where $\overline{\Gamma}_2 \in {\rm Div}_Q^{1,m-j-1}$ is a general curve on $Q$
passing through $p_2$ and intersecting $l_1$ at a point $q_1$.

Now, the curve $[\Gamma_1 + l_1] \in {\rm Div}_X^{1,j+1}$. Since, by induction and by \eqref{eq:div1m},
$${\rm dim}({\rm Div}_X^{1,j+1})
= {\rm dim} ({\rm Div}_X^{1,j}) + 2$$we can find a one-dimensional family of curves in ${\rm Div}_X^{1,j+1}$ passing through
$q_1$ and $p_2$, and degenerating to $\Gamma_1 + l_1$. This proves the assertion: a connecting point of
${\rm Div}_T^{1,m}(j)$ and ${\rm Div}_T^{1,m}(j+1)$ is
$$\Gamma_1 \cup (\overline{\Gamma}_2 + l_1) = (\Gamma_1 + l_1) \cup \overline{\Gamma}_2.$$
\end{proof}

Next, we need the following:

\begin{claim}\label{cl:gammaT}
Any $[\Gamma_T] \in {\rm Div}_T^{1,m}$ is connected and {\em non-special},
i.e.\ $h^1(\Oc_{\Gamma_T}(1)) = 0$.
Therefore
\begin{equation}\label{eq:C0}
h^0(\Oc_{\Gamma_T}(1)) = m - g + 1.
\end{equation}
\end{claim}

\begin{proof}[Proof of Claim \ref{cl:gammaT}] Consider the exact sequence
$$0 \to \Oc_T(H-\Gamma_T) \to \Oc_T (H) \to \Oc_{\Gamma_T}(H) \to 0.$$
From the definition of $T$, it follows that
$h^1(\Oc_T (H)) = 0$ because $T$ is linearly normal in $\Pp^R$.

To prove non-speciality of $\Gamma_T$, we will prove $h^2(\Oc_T(H-\Gamma_T)) = 0$.
We recall that $T$ is a connected union of two smooth, irreducible surfaces, with normal crossings,
so the dualizing sheaf of $T$ is associated to
a Cartier divisor, denoted by $K_T$; by Serre duality, we need to compute
$h^0(\Oc_T(K_T - H + \Gamma_T))$.

Since$$ \Oc_T(K_T - H + \Gamma_T) \hookrightarrow \Oc_X(K_T - H + \Gamma_T) \oplus \Oc_Q(K_T - H + \Gamma_T),$$
it suffices to prove $h^0(\Oc_T(K_T - H + \Gamma_T)|_{\Sigma}) = 0$, where $\Sigma$ is either $Q$ or $X$.

In the former case, if $l$ and $r$ denote the two distinct rulings on $Q$, we get
$$h^0(\Oc_T(K_T - H + \Gamma_T)|_Q) = h^0( \Oc_Q((m - m_X - 2)l - 2 r)) = 0 ,$$since
$ H|_Q \sim l + r$, $K_T|_Q \sim K_Q + l_1 + l_2$ and  $\Gamma_T|_Q \sim r + (m - m_X - 1) l$.
With similar computations, if $f$ denotes the ruling of $X$, we obtain
$$h^0(\Oc_T(K_T - H + \Gamma_T)|_X) = h^0(\Oc_X (- 2H + (m + 2g-2) l)) = 0,$$which
proves non-speciality of $\Gamma_T$.

Since $\Gamma_T$ is an effective Cartier divisor on $T$, from the exact sequence
$$0 \to \Oc_T (- \Gamma_T) \to \Oc_T \to \Oc_{\Gamma_T} \to 0 $$and from analogous
computations as above, one shows that $h^1(\Oc_T (- \Gamma_T)) =
h^1(\Oc_T (K_T + \Gamma_T) ) = 0$, so $h^0(\Oc_{\Gamma_T}) = h^0(\Oc_T) = 1$.
By Riemann-Roch theorem on $\Gamma_T$, this gives \eqref{eq:C0}.
\end{proof}

Let $h := m-g$; from the Euler sequence restricted to $\Gamma_T$
and from Claim \ref{cl:gammaT}, we have $h^1(\N_{\Gamma_T/\Pp^h})=0$.

From the inclusions $\Gamma_T \subset \Pp^h \subset \Pp^R$, we have the exact sequence:
$$0 \to \N_{\Gamma_T/\Pp^h} \to \N_{\Gamma_T/\Pp^R} \to \bigoplus_{d-g+1-m} \Oc_{\Gamma_T}(H)
\to 0,$$which shows that also $h^1(\N_{\Gamma_T/\Pp^R})=0$.

Observe that $\Gamma_T$ is l.c.i.\ in $T$ and that $T$ is l.c.i.\ in
$\Pp^R$, i.e.\ $\Gamma_T \subset T$ and $T \subset
\Pp^R$ are {\em regular embeddings} (cf.\ e.g.\ \cite{Ser}).
From  \cite[Proposition 4.5.3]{Ser}, the pair $(T,\Gamma_T)$
corresponds to a smooth point of the Flag-Hilbert scheme $F_{d,g,m}$ parametrizing pairs $(S, \Gamma_S)$, with
$[S] \in {\HH}_{d,g}$ and $[\Gamma_S] \in {\rm Div}_S^{1,m}$
(cf.\ \cite[\S\,4.5]{Ser}), since the obstructions are contained in $$
H^1(\N_{\Gamma_T/\Pp^R}) \times_{H^1(\N_{T/\Pp^R}
\otimes \Oc_{\Gamma_T})} H^1(\N_{T/\Pp^R})= (0).$$

Let us consider the projection
\begin{equation}\label{eq:p1}
F_{d,g,m} \stackrel{\pi}{\longrightarrow} {\HH}_{d,g}.
\end{equation}For any $[S] \in {\HH}_{d,g}$, the fibre of $\pi$ over $[S]$ with its scheme structure,
coincides with ${\rm Div}_S^{1,m}$. As we saw, the fibre over $[T]$, i.e.\ ${\rm Div}_T^{1,m}$ as in \eqref{eq:DT},
consists of smooth points of $F_{d,g,m}$. Hence, there is a non-empty, open Zariski subset $U \subset \HH_{d,g}$ such that
$\pi^{-1}(U)$ consists of smooth points of $F_{d,g,m}$. By generic smoothness, it follows that the
fibre over the general point  $[S] \in U$, which is ${\rm Div}_S^{1,m}$, is smooth.

We are left to show that ${\rm Div}_S^{1,m}$ is connected, for $d_m >0$ and for $[S] \in {\HH}_{d,g}$ general. To this aim, consider the morphism $\pi$ as in \eqref{eq:p1}. We need the following

\begin{claim}\label{cl:genred} ${\rm Div}_T^{1,m}$ is {\em generically reduced}, i.e.\ it is reduced at the
general point of any irreducible component of dimension $d_m$.
\end{claim}

\begin{proof}[Proof of Claim \ref{cl:genred}] To prove this, since $[T] \in \HH_{d,g}$ is a smooth point, we have
to show that if $\Gamma_T$ is a general point of a component of ${\rm Div}_T^{1,m}$, then the differential of
the map $\pi$ at the point $(T, \Gamma_T)$ is surjective. From \cite[\S\,4.5]{Ser}, and from the diagram
\[
\xymatrix@R-2mm@C-2mm{   &  &  &  0 \ar[d]  \\
  &  &  &  \N_{T/\Pp^R} (- \Gamma_T) \ar[d] \\
  &  &  &  \N_{T/\Pp^R} \ar[d] \\
0 \ar[r]  & \N_{\Gamma_T/T}  \ar[r]  & \N_{\Gamma_T/\Pp^R} \ar[r]
   & \N_{T/\Pp^R}|_{\Gamma_T} \ar[r] \ar[d] & 0 \\
  &  &  &  0 }
\]
it suffices to show that $h^1(\N_{\Gamma_T/T}) = 0$.
We prove this by induction on $g$.  Note that the case $g=0$ is trivially true.

A general point of an irreducible component of ${\rm Div}_T^{1,m}$
is a section $\Gamma_T$ of $T$, as in \eqref{eq:uniT}, from which we keep notation.

Since $\Gamma_T$ is a Cartier divisor, $\N_{\Gamma_T/T} = \Oc_{\Gamma_T}(\Gamma_T)$. Let $\nu : C \to \Gamma_T$
be the normalization morphism. Then, $C$ is the disjoint union of two smooth, irreducible curves $C = C_X \cup C_Q$,
where $C_X \cong \Gamma_X$ and $C_Q \cong \Gamma_Q$. One has the standard exact sequence
\begin{equation}\label{eq:nu}
0 \to \Oc_{\Gamma_T} \to \nu_*(\Oc_C) \to \Oc_{\{p_1, p_2\}} \to 0.
\end{equation}If we tensor \eqref{eq:nu} with $\Oc_{\Gamma_T}(\Gamma_T)$, we get
\begin{equation}\label{eq:nugamma}
0 \to \Oc_{\Gamma_T}(\Gamma_T) \to \nu_*(\Oc_C) \otimes \Oc_{\Gamma_T}(\Gamma_T) \to \Oc_{\{p_1, p_2\}} \to 0.
\end{equation}Since $\nu$ is a finite morphism,
$$H^1(\nu_*(\Oc_C) \otimes \Oc_{\Gamma_T}(\Gamma_T)) \cong H^1(\Oc_C(\nu^*(\Gamma_T))).$$Thus,
$$H^1(\Oc_C(\nu^*(\Gamma_T))) \cong H^1(\Oc_{C_X}(C_X)) \oplus H^1(\Oc_{C_Q}(C_Q)) \cong
H^1(\N_{\Gamma_X/X}) \oplus H^1(\N_{\Gamma_Q/Q}).$$By induction, this is zero.

By \eqref{eq:nugamma}, to prove $h^1(\N_{\Gamma_T/T}) = 0$ it suffices to prove that the map
\begin{equation}\label{eq:ro2}
H^0(\nu_*(\Oc_C) \otimes \Oc_{\Gamma_T}(\Gamma_T)) \stackrel{\rho}{\to} H^0(\Oc_{\{p_1,p_2\}}) \cong \C^2
\end{equation}is surjective.

As above, $H^0(\Oc_C(\nu^*(\Gamma_T)))  \cong
H^0(\N_{\Gamma_X/X}) \oplus H^0(\N_{\Gamma_Q/Q})$ and the map $\rho$ is given by
\begin{equation}\label{eq:ro2bis}
\rho((\sigma_1, \sigma_2)) = (\sigma_1(p_1) -\sigma_2(p_1), \sigma_1(p_2) -\sigma_2(p_2)).
\end{equation}

Since, by assumption, $d_m >0$, two cases have to be considered, as in the proof of Claim \ref{cl:cod}. If $m - m_X > 1$, by the expression \eqref{eq:ro2bis}, $\rho$ is clearly surjective. If $m - m_X = 1$, by induction $d_{m_X} = d_m + 1 \geq 2$; thus, also in this case the map $\rho$ is surjective,  because by the genericity of the lines $l_1$ and $l_2$ and of the quadric $Q$, the points $p_1$ and $p_2$ are general on $T$, hence they give independent conditions to the curves in ${\rm Div}_X^{1,m_X}$. This ends the proof of the claim.
\end{proof}

We claim further that ${\rm Div}_T^{1,m}$ is actually reduced.
Indeed, since $F_{d,g,m}$ is smooth along the fibre of $\pi$ over $[T]$, this fibre
is locally complete intersection
in $F_{d,g,m}$, hence it has no embedded component. Thus, being generically
reduced and with no embedded component, this fibre, isomorphic to ${\rm Div}_T^{1,m}$,
is reduced.

Since ${\rm Div}_T^{1,m}$ is connected, by Claim \ref{cl:conn}, and reduced, then
$h^0(\Oc_{{\rm Div}_T^{1,m}}) = 1$. Finally, the connectedness of ${\rm Div}_S^{1,m}$, for $[S] \in {\HH}_{d,g}$ general,
follows by the flatness of $\pi$ over $[T]$ and by semicontinuity.
\end{proof}

\begin{remark}\label{rem:impo} The proof of Theorem \ref{thm:ganzo2} also shows that
${\rm Div}_T^{1,m}$ is the flat limit of ${\rm Div}_S^{1,m}$.
\end{remark}

\begin{remark} Observe that Theorem \ref{thm:ganzo2} implies that, if $[S] \in \HH_{d,g}$ is general and $[\Gamma] \in {\rm Div}_S^{1,m}$, then
\begin{equation}\label{eq:smgamma}
h^1(\N_{\Gamma/S}) = 0.
\end{equation}Indeed, $h^0(\N_{\Gamma_S/S}) = d_m$ by smoothness and  \eqref{eq:smgamma} follows by
Riemann-Roch.

\end{remark}

\section{Some enumerative results}\label{S:enumerative}

The aim of this section is to prove some enumerative results concerning the scheme
${\rm Div}_S^{1,m}$, for $[S] \in \HH_{d,g}$ general, with $d$ and $g$ as in Theorem \ref{thm:Lincei}.

\subsection{Ghione's theorem}\label{SS:1}

In  \cite[Th\'eor\`eme 6.4 and 6.5]{Ghio} Ghione proves some basic enumerative properties concerning
unisecant curves of a {\em general} ruled surfaces of degree $d$ and genus $g$,
which were originally stated by C. Segre (cf.\ \cite{Seg}). According to Theorem \ref{thm:Ghio}, Ghione's results
are asymptotical, i.e.\ they apply to scrolls of sufficiently high degree $d$.
Theorem \ref{thm:ganzo2} allows us to prove a more precise statement for ruled surfaces over
a curve with general moduli.

\begin{theorem}\label{thm:ganzo}
Let $g \geq 0$ and $d$ be as in Theorem \ref{thm:Lincei}. Let $[S]
\in \HH_{d,g}$ be the general point. Let $\overline{m} := \lfloor \frac{d+g}{2}\rfloor$.
Then the minimal degree of the  unisecant curves of  $S$ is
\begin{itemize}
\item[(1)] $\frac{d+g-1}{2}$ if $d+g$ is odd; moreover
there are $2^g$ such sections;
\item[(2)] $\frac{d+g}{2}$ if $d+g$ is even; moreover there is a smooth, irreducible, one-dimensional
family of such sections.
\end{itemize}
\end{theorem}

\begin{proof} The case $g=0$ is well known: since $[S] \in \HH_{d,0}$
is general, then $S \subset \Pp^{d+1}$ is a
smooth, {\em balanced} rational normal scroll of degree $d$
(cf.\ e.g.\ \cite[Proposition 3.8]{CCFMLincei}).

For $g \geq 1$, apply Theorems \ref{thm:scrollfib1},  \ref{thm:ganzo2} and \cite[Th\'eor\`eme 6.4 and 6.5]{Ghio}.
\end{proof}


As a byproduct of the proof of Proposition \ref{prop:monodromy} below,
we will give an alternative proof of the enumerative part of Theorem \ref{thm:ganzo} (cf.\ Remark \ref{rem:casa05}).

\subsection{The index computation}\label{SS:2}

Let $g \geq 0$ and $d$ be as in Theorem \ref{thm:Lincei} and let $[S]
\in \HH_{d,g}$ be the general point. Let $m > \frac{d+g}{2}$ be an integer.
Then, ${\rm dim}({\rm Div}_S^{1,m}) = d_m >0$ (see \eqref{eq:div1m}).

The {\em index} of ${\rm Div}_S^{1,m}$ is the number of curves in  ${\rm Div}_S^{1,m}$ passing through $d_m$
general points on $S$ (cf.\ e.g.\ \cite[\S\;11, p.\ 137]{Seg}).

The following result (see \cite[\S\;4, p.\ 132, \S\,12, 13, pp.\ 138--140]{Seg}) can be derived
from Theorem \ref{thm:ganzo}.

\begin{proposition}\label{prop:index} Hypotheses as in Theorem \ref{thm:ganzo}.
Then, the index of ${\rm Div}_S^{1,m}$ is
$2^g$.
\end{proposition}

\begin{proof} The case $g=0$ is trivial, so assume $g \geq 1$. Let $\Lambda$ be a set of
$d_m$ general points on $S$ and let $\Gamma$ be any unisecant curve of $S$ of degree $m$ passing through
$\Lambda$. Denote by $\pi_{\Lambda} : S \to \Pp^{R'}$ the projection of $S \subset \Pp^R$ from $\Lambda$,
where $R = d-2g+1$.

If $S' : =  \pi_{\Lambda} (S) $, then $S' \subset \Pp^{R'}$ is a smooth, non-special scroll of genus
$g' = g$ such that:
\begin{itemize}
\item $d' : = \deg(S') = d - d_m = 2d + g - 2m -1$,
\item $R' = R - d_m = 2 d - 2 m - g = d' - 2 g' + 1$.
\end{itemize}The numerical assumptions of Theorem \ref{thm:Lincei} hold for $S'$.
Furthermore, by generality of  $[S] \in \HH_{d,g}$, also $[S'] \in \HH_{d',g'}$ is a general
point.

Denote by $\Gamma' = \pi_{\Lambda}(\Gamma)$. Then $[\Gamma'] \in {\rm Div}_{S'}^{1,m'}$, where
$m' = m - d_m = d + g - 1 -m$. Since$$d' + g' = d - d_m + g = 2(d + g - m) - 1$$is odd
then, for any $d$, $g$ and $m$ as above, ${\rm dim}({\rm Div}_{S'}^{1,m'}) = d_{m'} = 0$. Thus,
by Theorem \ref{thm:ganzo}, any such $\Gamma'$ is a section
of minimal degree of $S'$ and there exist $2^{g'} = 2^g$ such sections.

One concludes by observing that the correspondence between
elements of ${\rm Div}_S^{1,m}$ passing through $\Lambda$ and elements of  ${\rm Div}_{S'}^{1,m'}$
is bijective.
\end{proof}

\subsection{The monodromy action}\label{SS:3} In this subsection, we study the monodromy of the
$2^g$ sections of minimal degree $\overline{m}$
of a scroll $S$ corresponding to the general point of $\HH_{d,g}$, when $d_{\overline{m}} = 0$.

\begin{proposition}\label{prop:monodromy}
Let $g \geq 1$, $d$ and $\HH_{d,g}$ be as in Theorem \ref{thm:Lincei}. Assume $d+g$ odd and let
$\overline{m} = \frac{d+g-1}{2}$. Then, for $[S] \in \HH_{d,g}$ general,
the monodromy acts on  ${\rm Div}_S^{1,\overline{m}}$ as the full symmetric group.
\end{proposition}

\begin{proof} Let $[Y] \in \HH_{d,g}$ be as in Construction
\ref{const:Y}. As in the proof of Theorem \ref{thm:ganzo2},
a unisecant curve $\Gamma_Y$ of $Y$ will be a union of a unisecant curve $\Gamma_W$ of $W$ and unisecant
curves $\Gamma_j$ of the quadrics $Q_j$, for any $1 \leq j \leq g$,
i.e.\ $$\Gamma_Y = \Gamma_W \cup (\bigcup_{i=1}^g \Gamma_j),$$
with matching conditions on the rulings
$l_{1,j}, \, l_{2,j}$, for $1 \leq j \leq g$.

Since we want to consider unisecant curves of $S$, properly contained in its hyperplane section, it immediately follows
that any $\Gamma_j$ on $Q_j$ is either a conic, say $C_j$, or a line, say $r_j$, not belonging to the ruling
$|l_{i,j}|$, for $1 \leq i \leq 2$ and $1 \leq j \leq g$.

Up to a permutation of the indices, for some $0 \leq k \leq g$, we may assume we have conics$$C_j \subset Q_j, \;\; {\rm for} \; 1 \leq j \leq k \leq g$$and lines$$r_j \subset Q_j,  \;\; {\rm for} \; j >k,$$when $k \geq 1$, or
only lines, when $k =0$.

In any case,$$\Gamma_Y = \Gamma_W \cup (\bigcup_{j \leq k} C_j) \cup (\bigcup_{i=k+1}^g r_i),$$
therefore
$$\deg(\Gamma_W) := \nu_k = \frac{d+g-1}{2} - 2 k - (g-k) = \frac{d-g-1}{2} - k.$$Notice that
one has $\nu_k \geq \frac{d-2g-1}{2}$ which, by Theorem \ref{thm:ganzo},
is the minimal degree for the unisecant curves on $W$; hence
$g \geq 2k$. Moreover, the unisecant curves $\Gamma_W$ form a complete linear system $\Lambda_k$ on $W$ and,
by  \eqref{eq:div1m}, one has$${\rm dim}(\Lambda_k) = g - 2k.$$

Consider $$\bigcup_{1 \leq i \leq 2, \; 1 \leq j \leq g} l_{i,j} \;\; \cong \;\; (\Pp^1)^{2g}.$$For any $k \geq 0$,
we have a rational map
$$\lambda_k : \Lambda_k \dashrightarrow (\Pp^1)^{2g},$$
which is generically injective. Let
$V_k := {\rm Im}(\lambda_k)$; then $[V_k]$ is a cycle of dimension $g - 2k$
in the Chow ring of $(\Pp^1)^{2g}$.

Consider the projection to the $i$-th factor
$$(\Pp^1)^{2g} = (\Pp^1 \times \Pp^1)^g \stackrel{\pi_i}{\longrightarrow} (\Pp^1\times \Pp^1)_i.$$Set
$$H_i = \pi_i^*( \Oc_{\Pp^1 \times \Pp^1} (1)), \; 1 \leq i \leq g.$$Thus,
$|H_i|$ is base-point-free, for any $1 \leq i \leq g$. Therefore,
$${\rm dim}(V_k \cap (\bigcap_{t = k+1}^g H_t)) = {\rm max} \, \{-1, -k\},$$where
$H_{i}$ is general in $|H_i|$, for $ k+1 \leq i \leq g$.
By generality and by the matching conditions,
we need this dimension to be non-negative. This implies $k=0$.

In the above setting, define $V_{0,i}:=V_0 \cap  H_1 \cap \ldots \cap H_i$ for $i\geq 0$; with this notation
$V_{0,0}$ coincides with $V_0$. One has
$$v_{0,i}:={\rm dim}(V_{0,i})= {\rm max} \, \{-1,  g-2k-i\}.$$
We claim that $V_{0,i}$ is irreducible as soon as $v_{0,i} \geq 1$,
i.e.\ $g\geq 2k+i+1$. The assertion holds for $V_{0,0}=V_0$, so we proceed
by induction on $i$. Thus we assume $i\geq 1$,  $v_{0,i}\geq 1$ and
$V_{0,i-1}$ irreducible of dimension $v_{0,i-1} = v_{0,i}+1\geq 2$. Let us prove that
the projection  $\pi'_i:V_{0,i-1} \dashrightarrow (\Pp^1\times \Pp^1)_i$  is dominant.
To see this, let $\Lambda_{0,i-1}$ be the
pull--back via the map $\lambda_0$ of $V_{0,i-1}$. This is a sublinear system
of $\Lambda_0$ of dimension $v_{0,i}+1\geq 2$. If $\pi'_i$ were not dominant, then,
by our generality assumptions, the linear system $\Lambda_{0,i-1}$ of unisecants
would map two general lines of the scroll $W$ to the same line, hence $\Lambda_{0,i-1}$
would be a pencil, a contradiction.  Since $\pi'_i$ is dominant, the
claim follows by Bertini's theorem.

Finally,  $V_{0,g-1}$ is an irreducible curve and $\Lambda_{0,g-1}$ is a
pencil. The same argument as above shows that $\pi_g$  maps $V_{0,g-1}$
injectively to an irreducible curve $\Gamma$ on
the smooth quadric $Q_g \cong (\Pp^1\times \Pp^1)_g \subset \Pp^3$.
By \cite[Lemma, p.\ 111]{ACGH}, the monodromy group of this curve in $\Pp^3$ is the full symmetric group, implying the
assertion.
\end{proof}

\begin{remark}\label{rem:casa05} With the same ideas as in the proof
of Proposition \ref{prop:monodromy}, we can give
an alternative proof of the fact that, when  $d+g$ is odd, then $S$ contains
$2^g$ sections of minimal degree.  Indeed, since $[Y] \in \HH_{d,g}$ is a smooth point and
the map $\pi$ as in \eqref{eq:p1} is in this
case smooth over $[Y]$ (see the proof of Theorem \ref{thm:ganzo2}),
we see that $\deg({\rm Div}_S^{1, \frac{d+g-1}{2}}) = \deg(\Gamma)$, where $\Gamma$ is the curve on the smooth
quadric  $Q_g$ considered at the end of Proposition \ref{prop:monodromy}.

It is not difficult to see that $ \deg(\Gamma) = 2^g$. In fact, any
$H_i \subset (\Pp^1)^{2g}$ is linearly equivalent to $r_{1,i} + r_{2,i}$, where $r_{j,i}$, $ 1 \leq j \leq 2$,
are the pull-backs, via $\pi_i$, of the two rulings of $(\Pp^1 \times \Pp^1)_i$.
Therefore, in the Chow ring of $(\Pp^1)^{2g}$, one has
\begin{equation}\label{eq:(*)}
H_1 \cdot H_2 \cdot \ldots \cdot H_g =
\sum_{\underline{i}, \underline{j}} R_{\underline{i},\underline{j}},
\end{equation}where
$R_{\underline{i},\underline{j}}:= r_{1, i_1} \ldots r_{1, i_k} r_{2, j_1} \ldots r_{2, j_h}$ and
$i_1 < i_2 < \cdots < i_k$,
$j_1 < j_2 < \cdots < j_h$,  such that
$\{ i_1 , i_2 , \cdots , i_k, j_1 , j_2 , \cdots ,j_h \} = \{ 1 , \cdots , g \}$, in particular
$h + k = g$. For any summand, one has $$R_{\underline{i},\underline{j}} \cdot V_0 = 1.$$The intersection is
in fact equivalent to imposing to the linear system $\Lambda_0$, of dimension $g$, $g$ general points on $W$.
The assertion follows since the right hand side of \eqref{eq:(*)} contains $2^g$ summands.

\end{remark}

\subsection{The genus computation}\label{SS:5}

Let $g \geq 0$ and $d$ be as in Theorem \ref{thm:Lincei} and let $[S]
\in \HH_{d,g}$ be the general point. Let $m > \frac{d+g}{2}$ be an integer.
Then, ${\rm dim}({\rm Div}_S^{1,m}) = d_m >0$ (see \eqref{eq:div1m}).

Given $Z$ a general $0$-dimensional subscheme of $S$ of length $d_m-1$,
there is a $1$-dimensional family ${\mathcal D} \subset {\rm Div}_S^{1,m}$ consisting of curves
containing $Z$. The following proposition computes the genus of ${\mathcal D}$,
slightly extending the results in \cite{Gro} and \cite[Example 3.2]{Oxb}.

\begin{proposition}\label{prop:genus} Hypotheses as in Theorem \ref{thm:Lincei}.
Then, ${\mathcal D}$ is smooth and irreducible, of genus
$$\gamma:= 2^g (g-1) + 1.$$
\end{proposition}

\begin{proof} We prove the assertion in case $d_m=1$.  The case $d_m> 1$ can be dealt with
as in Proposition \ref{prop:index}.

As in the proof of Proposition \ref{prop:monodromy},
we can consider a cycle $V'_k$, $1 \leq k \leq g$,
in the Chow-ring of $(\Pp^1)^{2g}$, which is the image of the complete
linear system $\Lambda'_k$ on $W$ of unisecant curves of degree
$$\mu_k := \frac{d+g}{2} - 2k - (g-k) = \frac{d-g}{2} - k$$
via the obvious rational map $\lambda'_k:\Lambda'_k \dashrightarrow (\Pp^1)^{2g}$.
Then ${\rm dim}(\Lambda'_k) = {\rm dim}(V'_k) = g+1 - 2k$ and
$${\rm dim} (V'_k \cap (\bigcap_{t= k+1}^gH_t)) = {\rm max} \; \{ -1, 1-k\}.$$In the present situation,
we need the above dimension to be either $0$ or $1$, i.e.\ either $k =0$ or $k=1$.

\begin{itemize}
\item[(1)] If $k=0$, ${\rm dim}(V'_0) = g+1$ and the curve
$$\Xi_0 := V'_0 \cap (\bigcap_{j = 1}^g H_j),$$which is smooth and irreducible (see
proof of Proposition \ref{prop:monodromy}), is a component of the family of unisecant curves on $Y$.
\item[(2)] If $k =1$, on $W$ we have the linear system $\Lambda_1'$ of curves of degree
$\mu_1 = \frac{d-g-2}{2}$, where ${\rm dim}(\Lambda_1')= g-1$, and there exists an index $l \in \{1, \ldots , g\}$
such that on the quadric $Q_l$ we  have conics, whereas on the quadrics $Q_j$, for
$1 \leq j \neq l \leq g$, we have lines. In this case,
for any $ 1 \leq l \leq g$, we have a reduced, $0$-dimensional scheme
$$\Xi_{1,l} := V'_1 \cap (\bigcap_{1 \leq j \neq l \leq g} H_j).$$For each point in $\Xi_{1,l}$ we have a rational
component of the family of unisecant curves on $Y$.
\end{itemize}

First we compute the class of $\Xi_0$. As in Remark \ref{rem:casa05},
$$H_1 \cdot H_2 \cdot \ldots \cdot H_g = \sum_{\underline{i}, \underline{j}} R_{\underline{i},\underline{j}}.$$ Fix general points $p_i\in l_{1,i}$ and $q_i\in  l_{2,i}$. Consider $\underline{i}=(i_1,\ldots,i_k)$ and
$\underline{j}=(j_1\ldots,j_h)$ such that $\{i_1,\ldots,i_k,j_1,\ldots,j_h\}=\{1,\ldots,g\}$
and let $\Lambda'_{0,\underline{i},\underline{j}}$ be  the sublinear system of
$\Lambda'_0$ consisting of all curves containing $p_{i_1},\ldots, p_{i_k}, q_{j_1},\ldots, q_{j_h}$.
This is a pencil, whose image $\Xi_{\underline{i},\underline{j}}$  in $(\Pp^1)^{2g}$ has class
$R_{\underline{i},\underline{j}} \cdot V_0'$. Hence
$\Xi_0$ is homologous to $\sum_{\underline{i},\underline{j}} \Xi_{\underline{i},\underline{j}}$, i.e.\ to a sum of $2^ g$ copies of $\Pp^ 1$. It is not difficult to see how they intersect
each other. Indeed
$\Xi_{\underline{i},\underline{j}} \cdot \Xi_{\underline{i}',\underline{j}'}$ is non--zero, and it is
a point,  if and only if $\{p_{i_1},\ldots, p_{i_k}, q_{j_1},\ldots, q_{j_h}\}\cap \{p_{i_1},\ldots, p_{i_k'}, q_{j_1},\ldots, q_{j_h'}\}$ consists of exactly $g-1$ points. This implies that each
$\Xi_{\underline{i},\underline{j}}$ intersects exactly $g$ others $\Xi_{\underline{i}',\underline{j}'}$.

Fix now $ l$ any integer such that $1 \leq l \leq g$. Let $\underline{a}$ (resp., $\underline{b}$)
denote a sequence of integers $a _1 < a_2 < \cdots < a_k$ (resp., $b _1 < b_2 < \cdots < b_h$),
such that $h + k = g-1$ and $\{a _1 , a_2 \ldots , a_k , b _1 ,  b_2,  \ldots, b_h\} =
\{1, \ldots, g\} \setminus \{l\}$.
Let $D_{l,\underline{a},\underline{b}}$ be the unique curve of
$\Lambda'_1$ containing $p_{i_1},\ldots, p_{i_k}, q_{j_1},\ldots, q_{j_h}$.
This matches a pencil of conics on the quadric $Q_l$, to give a smooth rational component $\Xi'_{l,\underline{a},\underline{b}}$ of the family of unisecant curves on $Y$ as described in (2). These rational curves do not intersect each other. However, they intersect the curves $\Xi_{\underline{i},\underline{j}}$,
and precisely only two of them, i.e.\ the ones for which either $\underline{i}$ consists of $\underline{a}$ and $l$, or $\underline{j}$ consists of $\underline{b}$ and $l$.

As in the proof of Theorem \ref{thm:ganzo2}, we see that
the flat limit of the smooth curve
${\mathcal D}= {\rm Div}_S^{1,\frac{d+g}{2}}$ on $[S] \in \HH_{d,g}$ general
can be in turn flatly degenerated to a union
of smooth rational curves whose dual graph  $G$ has:
\begin{itemize}
\item $v=2^g +  2^{g-1}\cdot g = 2^{g-1} (2+g)$  vertices and
\item $e= g\cdot  \frac{2^g}{2} + 2\cdot  2^{g-1}\cdot g = 3\cdot 2^{g-1}\cdot g$ edges.
\end{itemize} Hence $\chi(G) = v - e = 2^g (1-g)$ and therefore the arithmetic genus of $\mathcal D$ is $\gamma=1-\chi(G)=1+ 2^g (g-1)$.
\end{proof}

Notice that, when $g=2$, then $\gamma=5$. This also follows, via a different approach, from
\cite[Remark 1.6]{Oxb}, because in this case the curve
${\rm Div}_S^{1,\frac{d+2}{2}}$ is isomorphic to a divisor in $|2 \Theta|$ in
the jacobian of the curve $C$.



\end{document}